\newif\ifsiopt
\sioptfalse
\ifsiopt
\documentclass[review,onefignum,onetabnum]{siamart190516}

\input{ex_shared}
\else
\documentclass{article}
\usepackage{amsmath}
\usepackage{times}
\usepackage{amssymb}
\usepackage{amsthm}
\usepackage{bm}
\usepackage{graphicx}
\usepackage{caption}
\usepackage[margin = 1in]{geometry}
\usepackage{cite}
\usepackage{color}
\usepackage[colorlinks,urlcolor=blue,citecolor=blue,linkcolor=blue]{hyperref}

\usepackage[colorinlistoftodos]{todonotes}

\newcommand{\innp}[1]{\left\langle #1 \right\rangle}

\newcommand{\mA}{\boldsymbol{A}}
\newcommand{\mI}{\boldsymbol{I}}

\newcommand{\zeros}{\textbf{0}}
\newcommand{\vx}{\boldsymbol{x}}
\newcommand{\vxb}{\boldsymbol{\bar{x}}}

\newcommand{\cF}{\mathcal{F}}

\newcommand{\cc}{\mathcal{C}}
\newcommand{\cA}{\mathcal{A}}

\newcommand{\vxh}{\boldsymbol{\hat{x}}}

\newcommand{\vyb}{\boldsymbol{\bar{y}}}

\newcommand{\vy}{\boldsymbol{y}}

\newcommand{\vv}{\boldsymbol{v}}

\newcommand{\vb}{\boldsymbol{b}}
\newcommand{\vg}{\boldsymbol{g}}
\newcommand{\vu}{\boldsymbol{u}}
\newcommand{\vub}{\bar{\boldsymbol{u}}}

\newcommand{\rr}{\mathbb{R}}

\theoremstyle{plain} \numberwithin{equation}{section}
\newtheorem{theorem}{Theorem}[section]
\numberwithin{theorem}{section}

\newtheorem{conjecture}{Conjecture}
\newtheorem{lemma}[theorem]{Lemma}

\newtheorem{fact}[theorem]{Fact}
\theoremstyle{definition}
\newtheorem{definition}[theorem]{Definition}

\newtheorem{remark}[theorem]{Remark}

\DeclareMathOperator*{\argmin}{argmin}

  \makeatletter
\newcommand{\subalign}[1]{%
  \vcenter{%
    \Let@ \restore@math@cr \default@tag
    \baselineskip\fontdimen10 \scriptfont\tw@
    \advance\baselineskip\fontdimen12 \scriptfont\tw@
    \lineskip\thr@@\fontdimen8 \scriptfont\thr@@
    \lineskiplimit\lineskip
    \ialign{\hfil$\m@th\scriptstyle##$&$\m@th\scriptstyle{}##$\hfil\crcr
      #1\crcr
    }%
  }%
}
\makeatother

\newcommand\norm[1]{\left\lVert#1\right\rVert}
\title{Potential Function-based Framework for Making the Gradients Small in Convex and Min-Max Optimization\thanks{This research was partially supported by the NSF grant CCF-2007757 and by the Office of the Vice
Chancellor for Research and Graduate Education at the University of Wisconsin–Madison with funding from the
Wisconsin Alumni Research Foundation. Part of this research was done while PW was attending UW-Madison as part of the Visiting International Student Program (VISP).}}
\author{Jelena Diakonikolas\\
Department of Computer Sciences\\
University of Wisconsin-Madison\\
\texttt{jelena@cs.wisc.edu}
\and
Puqian Wang\\
School of Mathematics\\
Shandong University\\
\texttt{e.puqian.wang@gmail.com}}
\date{}
\fi

\ifpdf
\hypersetup{
  pdftitle={Potential Function-based Framework for Making the Gradients Small in Convex and Min-Max Optimization},
  pdfauthor={J. Diakonikolas and P. Wang}
}
\fi

\begin{document}

\maketitle

\ifsiopt
\begin{keywords}
  gradient minimization, convergence analysis, potential function
\end{keywords}

\begin{AMS}
  90C06, 90C25, 65K05
\end{AMS}
\fi

\begin{abstract}
    Making the gradients small is a fundamental optimization problem that has eluded unifying and simple convergence arguments in first-order optimization, so far primarily reserved for other convergence criteria, such as reducing the optimality gap. We introduce a novel potential function-based framework to study the convergence of standard methods for making the gradients small in smooth convex optimization and convex-concave min-max optimization. Our framework is intuitive and it provides a lens for viewing algorithms that make the gradients small as being driven by a trade-off between reducing either the gradient norm or a certain notion of an optimality gap.  On the lower bounds side, we discuss tightness of the obtained convergence results for the convex setup and provide a new lower bound for  minimizing norm of cocoercive operators that allows us to argue about optimality of methods in the min-max setup.  
\end{abstract}

\section{Introduction}

One of the most basic facts in convex optimization is that a differentiable convex function attains its minimum at a point where its gradient equals zero, provided such a point exists. Thus, it is tempting to conclude that there is no difference between minimizing the function value or its gradient (in any suitable norm). This is only partially true, as we are almost never guaranteed to find a point at which the function is minimized; instead, we opt for a more modest goal of approximating such points. As it turns out, from an algorithmic point of view, there are major differences between guarantees provided for the function value (or optimality gap) and norm of its gradient. 

Much of the standard optimization literature on smooth (gradient-Lipschitz) convex first-order optimization has been concerned with providing guarantees for the optimality gap. There is comparatively much less work on guarantees for the norm of the gradient, most of it being initiated after the work of Nesterov~\cite{nesterov2012make}, which argued that such guarantees are natural and more informative than those based on the function value for certain linearly constrained optimization problems that frequently arise in applications. Further, unlike the optimality gap, which would require knowledge of the minimum function value to be usable as a stopping criterion, the norm of the gradient is readily available to the algorithm as a stopping criterion, as standard first-order methods define their iterates based on the gradient information. This insight is particularly useful for the design of parameter-free algorithms (i.e., algorithms that do not require knowledge of function parameters such as smoothness, strong convexity, or sharpness/constants of {\L}ojasiewicz inequality; see, e.g.,~\cite{lojasiewicz1963propriete,lojasiewicz1965ensembles,bolte2010characterizations,attouch2013convergence}), and as such has been used to design parameter-free algorithms that are near-optimal in terms of iteration complexity (i.e., optimal up to poly-logarithmic factors)~\cite{nesterov2013methods,lin2014adaptive,ito2019gradientmapping}. 

As for $L$-smooth functions the norm of the gradient can be bounded above as a function of the optimality gap $f(\vx) - f(\vx^*),$ where $\vx^* \in \argmin_{\vx} f(\vx)$, using 
\begin{equation}\label{eq:grad-norm-to-opt-gap}
    \frac{1}{2L}\|\nabla f(\vx)\|^2 \leq f(\vx) - f(\vx^*),
\end{equation} 
it is not surprising that convergence rates can be established for gradient norm minimization. What \emph{is} surprising, however, is that those rates can be faster than what is implied by Eq.~\eqref{eq:grad-norm-to-opt-gap} and existing results for convergence in function value/optimality gap. In particular, methods that are optimal in terms of iteration complexity for minimizing the optimality gap are not necessarily optimal for gradient norm optimization, and vice-versa. More specifically, the fast gradient method (FGM) of Nesterov~\cite{nesterov1983method} is iteration complexity-optimal for minimizing the optimality gap, but it is suboptimal for minimizing the gradient norm~\cite{kim2018generalizing,carmon2019lower}. 

More generally, the existing literature has not yet shed light on what is the basic mechanism that drives algorithms for gradient norm minimization. 
The only known iteration complexity-optimal algorithm for minimizing norm of the gradient of a smooth convex function is due to Kim and Fessler~\cite{kim2020optimizing}.\footnote{The optimality of the algorithm can be certified using the lower bound from~\cite{carmon2019lower}.} This algorithm was obtained by using the performance estimation framework of Drori and Teboulle~\cite{drori2014performance}, originally developed for understanding the worst-case performance of optimization algorithms. The algorithm~\cite{kim2020optimizing} itself and its convergence analysis are inferred from numerical solutions to a semidefinite program (SDP). As such, the intuition behind what is driving the convergence analysis of the algorithm and how the improved convergence rate is obtained is lacking, which constitutes an impediment to possibly generalizing this algorithm to other optimization settings. 

Even less is known in the setting of smooth convex-concave min-max optimization, where (near-)optimal convergence results have been established only recently~\cite{diakonikolas2020halpern,kim2019accelerated,lieder2020convergence,sabach2017first} and the problem has been much less studied from the aspect of oracle lower bounds~\cite{Ouyang2019,diakonikolas2020halpern,golowich2020last}. In particular, similar as in the case of convex optimization, classical methods for min-max optimization that are optimal for reducing the primal-dual gap, such as, e.g., the extragradient method~\cite{korpelevich1977extragradient}, mirror-prox~\cite{nemirovski2004prox}, and dual extrapolation~\cite{nesterov2007dual}, are suboptimal in terms of iteration complexity for minimizing the gradient norm. Interestingly, however, the methods that turn out to be (near-)optimal were originally studied in the context of fixed point iterations~\cite{krasnosel1955two,mann1953mean,halpern1967fixed}. 

In this paper, we introduce a novel potential function-based framework to study the convergence in gradient norm for smooth convex and convex-concave optimization problems. Our framework is intuitive, as it relies on establishing convergence of standard methods by interpreting it as a trade-off between reducing the gradient norm and reducing a notion of an optimality gap. The same view can be adopted in a unifying manner for methods such as  standard gradient descent, Nesterov FGM~\cite{nesterov1983method}, optimized method of Kim and Fessler~\cite{kim2020optimizing}, gradient descent-ascent (which is equivalent to Krasnosel'ski{\i}-Mann iteration~\cite{krasnosel1955two,mann1953mean}; see Section~\ref{sec:gda}), and Halpern iteration~\cite{halpern1967fixed}. We further complement these results with a discussion of optimality of the considered methods for convex optimization, and with a new lower bound for minimizing the norm of cocoercive operators (see Section~\ref{sec:prelims} for a precise definition and relationship to min-max optimization), which allows us to discuss optimality of gradient descent-ascent and Halpern iteration as methods for minimizing the gradient norm in smooth convex-concave min-max optimization.   


\subsection{Further Related Work} 

Understanding the phenomenon of acceleration and providing a unifying theory of first-order optimization algorithms has been an important topic in optimization research, with a flurry of recent research activity in this area~\cite{tseng2008,AO-survey-nesterov,attouch2000heavy,attouch2019rate,SuBC16,wibisono2016variational,wilson2016lyapunov,zhang2018direct,shi2018understanding,krichene2015accelerated,Scieur2017,Bubeck2015,drusvyatskiy2016optimal,lin2015universal,betancourt2018symplectic,AXGD,thegaptechnique,attouch2000heavy-esc,hu2017control,lessard2016analysis,song2021unified,diakonikolas2021generalized,attouch2020first}. However, the existing literature has almost exclusively focused on the optimality gap guarantees, with only a small subset of results seeking to provide guarantees for gradient norm and primarily addressing FGM-type algorithms with suboptimal rates~\cite{diakonikolas2021generalized,shi2018understanding,attouch2020first}. 

Complementary to the literature discussed above, whose focus has been on deriving intuitive convergence analysis frameworks, another line of work has focused on using the SDP-based performance estimation framework of Drori and Teboulle~\cite{drori2014performance} to investigate the worst-case performance of optimization algorithms~\cite{kim2018generalizing,kim2019accelerated,kim2020optimizing,lieder2020convergence,de2020worst,taylor2017exact}. Most relevant to our work among these results are:  \cite{kim2018generalizing}, which investigated the worst-case performance of FGM-type methods in terms of gradient norm minimization, \cite{kim2020optimizing}, which obtained the first (and so far, the only) iteration complexity-optimal algorithm for minimizing the gradient norm of smooth convex functions, and  \cite{lieder2020convergence}, which obtained a tight worst-case convergence bound for Halpern iteration. While the SDP-based approach used in this line of work is useful for understanding the worst-case performance of existing algorithms (and even obtaining new algorithms~\cite{kim2020optimizing}), its downside is that, because the convergence arguments are computer-assisted (namely, they are inferred from numerical solutions to SDPs), they are generally not suitable for developing intuition about what is driving the methods and their analysis. Our work fills this gap by providing intuitive convergence proofs based on potential function arguments.

\subsection{Notation and Preliminaries}\label{sec:prelims}
Throughout the paper, we consider the Euclidean space $(\rr^d, \|\cdot\|),$ where $\|\cdot\| = \sqrt{\innp{\cdot, \cdot}}$ is the Euclidean norm and $\innp{\cdot, \cdot}$ denotes any inner product on $\rr^d$. We use $\{A_k\}_{k\geq 0}$ and $\{B_k\}_{\geq 0}$ to denote sequences of nondecreasing nonnegative numbers, and define $a_0 = A_0,$ $a_k = A_{k} - A_{k-1}$ for $k \geq 1,$ and, similarly, $b_0 = B_0,$ $b_k = B_{k} - B_{k-1}$ for $k \geq 1.$ 

We consider two main problem setups: (i) making the gradients small in convex optimization, and (ii) making the gradients small in min-max optimization.

\paragraph{Convex optimization.} 
In the first setup, we assume we are given first-order oracle access to a convex continuously differentiable function $f: \rr^d \to \rr.$ The first-order definition of convexity then applies, and we have:
\begin{equation*}
    (\forall \vx, \vy \in \rr^d): \quad f(\vy) \geq f(\vx) + \innp{\nabla f(\vx), \vy - \vx}.
\end{equation*}
We further assume that $f$ is $L$-smooth, i.e., that its gradients are $L$-Lipschitz continuous:
\begin{equation*}
    (\forall \vx, \vy \in \rr^d): \quad \|\nabla f(\vx) - \nabla f(\vy)\| \leq L\|\vx - \vy\|.
\end{equation*}
Recall that smoothness of $f$ implies:
\begin{equation}\label{eq:smooth-ub}
    (\forall \vx, \vy \in \rr^d):\quad f(\vy) \leq f(\vx) + \innp{\nabla f(\vx), \vy - \vx} + \frac{L}{2}\|\vy - \vx\|^2.
\end{equation}

The goal of the first setup is to, given $\epsilon > 0,$ construct a point $\vx$ such that
 $   \|\nabla f(\vx)\| \leq \epsilon$
in as few iterations (oracle queries to the gradient of $f$) as possible.
A useful fact that turns out to be crucial for the analysis in the convex case is the following (see, e.g.,~\cite[Section 3.5]{Zalinescu:2002}).
\begin{fact}\label{fact:smooth+convex}
A continuously differentiable function $f:\rr^d \to \rr$ is $L$-smooth and convex if and only if
\begin{equation}\label{eq:smooth+cvx}
    (\forall \vx, \vy \in \rr^d):\quad \frac{1}{2L}\|\nabla f(\vy) - \nabla f(\vx)\|^2 \leq f(\vy) - f(\vx) - \innp{\nabla f(\vx), \vy - \vx}.
\end{equation}
\end{fact}
Observe that  Fact~\ref{fact:smooth+convex} fully characterizes the class of smooth convex functions, and, as such, should be sufficient for analyzing any algorithm that addresses problems from this class. 
An immediate consequence of Fact~\ref{fact:smooth+convex} is that the gradient of a smooth convex function is cocoercive, i.e.,
\begin{equation}
    \innp{\nabla f(\vx) - \nabla f(\vy), \vx - \vy} \geq \frac{1}{L}\|\vx - \vy\|^2.
\end{equation}

\paragraph{Min-max optimization.} In the second setup, we are given oracle access to gradients of a function $\phi: \rr^{d_1}\times \rr^{d_2}\to \rr$, where $d_1 + d_2 = d.$ Function $\phi(\vx, \vy)$ is assumed to be convex-concave: convex in the first argument ($\vx$) when the second argument ($\vy$) is fixed and concave in the second argument ($\vy$) when the first argument ($\vx$) is fixed, for any values of $\vx, \vy$. Similar to the case of convex optimization, the goal in this case is, given $\epsilon > 0,$ to find a pair of points $(\vx, \vy) \in \rr^{d_1}\times \rr^{d_2}$ such that $\|\nabla \phi(\vx, \vy)\| \leq \epsilon$ in as few iterations (oracle queries to the gradient of $\phi$) as possible.

We consider the problem of minimizing the norm of the gradient of $\phi$ as the problem of minimizing the norm of the operator $F(\vu) = \big[\subalign{\nabla_{\vx}\phi(\vx, \vy)\\ -\nabla_{\vy} \phi(\vx, \vy)}\big],$ where $\vu = \big[\subalign{\vx \\ \vy}\big].$ When $\phi$ is convex-concave, $F$ is monotone, i.e., it holds
\begin{equation}\label{eq:mon-op}
    (\forall \vu, \vv \in\rr^d):\quad \innp{F(\vu) - F(\vv), \vu - \vv} \geq 0. 
\end{equation}
We will assume throughout that $F$ is $\frac{1}{L}$-cocoercive, i.e., that
\begin{equation}\label{eq:cocoercive-op}
    (\forall \vu, \vv \in\rr^d):\quad \innp{F(\vu) - F(\vv), \vu - \vv} \geq \frac{1}{L}\|F(\vu) - F(\vv)\|^2.
\end{equation}
Cocoercivity of $F$ implies that it is monotone and $L$-Lipschitz. The opposite does not hold in general, unless $F$ is the gradient of a smooth convex function (as we saw in the case of convex optimization described earlier). Nevertheless, cocoercivity is sufficient to capture the main algorithmic ideas of smooth min-max optimization, and the extensions to general smooth min-max optimization are possible through the use of approximate resolvent operators (see, e.g.,~\cite{diakonikolas2020halpern}). Further, it suffices to consider unconstrained problems, as extensions to constrained optimization problems are possible in a straightforward manner using a notion of \emph{operator mapping} (see, e.g.,~\cite{diakonikolas2020halpern}, where a similar idea was used). 

We assume here that there exists a point $\vu^* \in \rr^d$ such that $F(\vu^*) = 0.$ Due to cocoercivity of $F$ (Eq.~\eqref{eq:cocoercive-op}), this assumption implies that
\begin{equation}\label{eq:minty-gap}
   (\forall \vu \in \rr^d):\quad \innp{F(\vu), \vu - \vu^*} \geq \frac{1}{L}\|F(\vu)\|^2. 
\end{equation}
It will be useful to think of $\innp{F(\vu), \vu - \vu^*}$ as a notion of ``optimality gap'' for min-max optimization problems, as, using convexity-concavity of $\phi$, we have
$$
    (\forall (\vx, \vy) \in \rr^{d_1}\times \rr^{d_2}):\quad \phi(\vx, \vy^*) - \phi(\vx^*, \vy^*) + \phi(\vx^*, \vy^*) - \phi(\vx^*, \vy) \leq \innp{F(\vu), \vu - \vu^*}. 
$$

\section{Small Gradients in Convex Optimization}\label{sec:cvx-opt}
In this section, we consider the problem of minimizing  norm of the gradient of a smooth convex function. We show that all standard methods, including standard gradient descent, fast gradient method of Nesterov~\cite{nesterov1983method}, and the optimized gradient method of Kim and Fessler~\cite{kim2020optimizing}, can be captured within an intuitive potential function-based framework, where the progress of a method is established through a trade-off between the norm of the gradient and the optimality gap. Further, the complete convergence analysis of each of the methods can be fully carried out using only the cocoercivity inequality from Eq.~\eqref{eq:smooth+cvx}, which fully characterizes the class of smooth convex functions.  
\subsection{Gradient Descent}
As a warmup, we start by considering $L$-smooth but possibly nonconvex objectives $f.$ In this case, all that can be said about $f$ is that its gradients are $L$-Lipschitz, which implies Eq.~\eqref{eq:smooth-ub}. Further, for any method that does not converge to local maxima (unless initialized at one), we cannot hope to bound the norm of the last gradient -- all that we can hope for is the average or the minimum over all seen gradients. The simplest way to see this is by considering the one dimensional case: if the function is locally concave and the algorithm moves in the direction that reduces the function value, the absolute value of the function derivative must increase. 

Thus, assuming that the function is bounded below by some $f_{\star}>-\infty$, it is natural to consider methods that in each iteration either reduce the function value or the norm of the gradient. Such methods ensure that, $\forall k \geq 0:$
$$  
    a_k \|\nabla f(\vx_k)\|^2 + f(\vx_{k+1}) - f(\vx_k) \leq 0, 
$$
or, equivalently, that the following potential function
\begin{equation}
    \cc_k = \sum_{i=0}^k a_i \|\nabla f(\vx_i)\|^2 + f(\vx_{i+1})
\end{equation}
is non-increasing, where $a_i$ is some sequence of positive numbers. Equivalently, such methods ensure that $a_k \|\nabla f(\vx_k)\|^2 + f(\vx_{k+1}) - f(\vx_k) \leq 0,$ $\forall k \geq 0.$ 

As the only assumption we are making about $f$ is that it is $L$-smooth, the most we can do to bound $f(\vx_{k+1}) - f(\vx_k)$ is use Eq.~\eqref{eq:smooth-ub}. The tightest bound on $f(\vx_{k+1}) - f(\vx_k)$ that can be obtained from Eq.~\eqref{eq:smooth-ub} is attained when $\vx_{k+1} = \vx_k - \frac{1}{L}\nabla f(\vx_k)$ (i.e., for the standard gradient descent step) and is given by $f(\vx_{k+1}) - f(\vx_k) \leq - \frac{1}{2L}\|\nabla f(\vx_k)\|^2$, in which case the largest $a_k$ we can choose is $a_k = \frac{1}{2L}.$ As $\cc_k$ is non-increasing, it follows that $\cc_k \leq \cc_0,$ and we recover the familiar convergence bound of gradient descent:
\begin{equation}\label{eq:ncvx-gd-conv-bnd}
    \frac{1}{k+1}\sum_{i=0}^k \|\nabla f(\vx_i)\|^2 \leq \frac{2L(f(\vx_0) - f(\vx_{k+1}))}{k+1} \leq \frac{2L(f(\vx_0) - f_{\star})}{k+1}.
\end{equation}

When considering the case of a convex objective function $f,$ the first question to ask is how would convexity help to improve the bound from Eq.~\eqref{eq:ncvx-gd-conv-bnd}. The first observation to make is that Fact~\ref{fact:smooth+convex} fully characterizes the class of smooth convex functions, and, thus, Eq.~\eqref{eq:smooth+cvx} should be enough to carry out the analysis of any algorithm for smooth convex functions. 

Given that the function is convex, in this case it seems reasonable to hope that we can obtain a bound on the gradient norm at the last iterate. Thus, we could consider a potential function of the form
$$
    \cc_k = A_k \|\nabla f(\vx_k)\|^2 + f(\vx_k)
$$
and try enforcing the condition that $\cc_k \leq \cc_{k-1}$ for $A_k$ that grows as fast as possible with the iteration count $k.$  This approach precisely gives the bound $\|\nabla f(\vx_k)\|^2 \leq \frac{2L (f(\vx_0) - f(\vx^*))}{2k + 1},$ which is tight (see, e.g.,~\cite[Lemma~5.2]{kim2020optimizing}).

\begin{lemma}[Convergence of Gradient Descent]\label{lem:gd-cvx}
Let $f:\rr^d \to \rr$ be an $L$-smooth function that attains its minimum on $\rr^d$ and let $\vx^* \in \argmin_{\vx \in \rr^d}f(\vx)$. Let $\vx_0 \in \rr^d$ be an arbitrary initial point and assume that the sequence $\{\vx_k\}_{k \geq 0}$ evolves according to the standard gradient descent, i.e., $\vx_{k+1} = \vx_k - \frac{1}{L}\nabla f(\vx_k),$ $\forall k \geq 0$. Then
$$
    \cc_k = \frac{k}{L}\|\nabla f(\vx_k)\|^2 + f(\vx_k)
$$
is non-increasing with $k,$ and we can conclude that, $\forall k \geq 0$
$$
\|\nabla f(\vx_k)\|^2 \leq \frac{2L (f(\vx_0) - f(\vx^*))}{2k + 1}.
$$
\end{lemma}
\begin{proof}
We start by showing that $\cc_{k+1}\leq \cc_k,$ $\forall k \geq 0.$ By the definition of $\cc_k,$
$$
    \cc_{k+1} - \cc_k \leq \frac{k+1}{L}\|\nabla f(\vx_{k+1})\|^2 - \frac{k}{L}\|\nabla f(\vx_k)\|^2 + f(\vx_{k+1}) - f(\vx_k).
$$
Applying Fact~\ref{fact:smooth+convex} with $\vx = \vx_{k+1} = \vx_k - \frac{1}{L}\nabla f(\vx_k)$ and $\vy = \vx_k$, it follows that $f(\vx_{k+1}) - f(\vx_k) \leq -\frac{1}{2L}\|\nabla f(\vx_{k+1})\|^2 - \frac{1}{2L}\|\nabla f(\vx_k)\|^2,$ and, thus
$$
    \cc_{k+1} - \cc_k \leq  \frac{2k+1}{2L}\|\nabla f(\vx_{k+1})\|^2 - \frac{2k + 1}{2L}\|\nabla f(\vx_k)\|^2.
$$
To complete the proof that $\cc_{k+1}\leq \cc_k,$ it remains to argue that $\|\nabla f(\vx_{k+1})\| \leq \|\nabla f(\vx_k)\|,$ $\forall k \geq 0.$  This is clearly true if $\|\nabla f(\vx_{k+1})\| = 0,$ so assume $\|\nabla f(\vx_{k+1})\| \neq 0$. Applying Eq.~\eqref{eq:cocoercive-op} with $\vx = \vx_{k+1} = \vx_k - \frac{1}{L} \nabla f(\vx_k)$, $\vy = \vx_k$, and simplifying, it follows that:
$$
    \|\nabla f(\vx_{k+1})\|^2 \leq \innp{\nabla f(\vx_{k+1}), \nabla f(\vx_k)} \leq \|\nabla f(\vx_{k+1})\| \|\nabla f(\vx_k)\|,
$$
where the last inequality is by Cauchy-Schwarz. To conclude that $\|\nabla f(\vx_{k+1})\| \leq \|\nabla f(\vx_k)\|,$ it remains to divide both sides of the last inequality by $\|\nabla f(\vx_{k+1})\|.$

From the first part of the proof, it follows that $\cc_k \leq \cc_0,$ and, thus
$$
    \frac{k}{L}\|\nabla f(\vx_k)\|^2 \leq f(\vx_0) - f(\vx_k) = f(\vx_0) - f(\vx^*) + f(\vx^*) - f(\vx_k).
$$
It remains to observe that $f(\vx^*) - f(\vx_k) \leq - \frac{1}{2L}\|\nabla f(\vx_k)\|^2,$ which follows by applying Fact~\ref{fact:smooth+convex} with $\vx = \vx^*,$ $\vy = \vx_k$, and rearrange.
\end{proof}

\subsection{Methods that are Faster than Gradient Descent}

The potential functions we have seen so far (for gradient descent) trade off the gradient norm (squared) with the function value. Equivalently, we can view them as trading off the gradient norm with the optimality gap $f(\vx_k) - f(\vx^*),$ as $f(\vx^*)$ would cancel out in the analysis and the same argument would go through.

It is reasonable to ask whether we can obtain faster algorithms by using a different trade off, say, by considering potential functions of the form $\cc_k = A_k \|\nabla f(\vx_k)\|^2 + B_k (f(\vx_k) - f(\vx^*))$ or $\cc_k = \sum_{i=0}^k a_i \|\nabla f(\vx_i)\|^2 + B_k (f(\vx_k) - f(\vx^*)),$ where $B_k$ is some positive function of the iteration count $k$. 

Observe that for non-constant $B_k,$ one way or another, we would need to account for $\vx^*,$ which is not known to the algorithm. However, there are at least two ways around this issue. The first one is to utilize Eq.~\eqref{eq:smooth+cvx} to bound below $f(\vx^*)$. 
This approach does not lead to the optimal iteration complexity, but improves the overall bound compared to gradient descent and recovers a variant of Nesterov FGM. The second approach is to replace the optimality gap with a gap to some reference point. In particular, as we show below, optimized gradient method~\cite{kim2020optimizing} can be viewed as using the final point of the algorithm $\vx_N$ as the reference (or anchor) point. 

\subsubsection{Fast Gradient Method}

We start by considering a potential function that offers a different trade-off between the norm of the gradient and the optimality gap, defined by
\begin{equation}\label{eq:C_k_FGM}
    \cc_k = \sum_{i=0}^{k-1} a_i \|\nabla f(\vx_i)\|^2 + B_k(f(\vx_k) - f(\vx^*)),
\end{equation}
where $a_i > 0,$ $\forall i \geq 0$ and the sequence of scalars $B_k > 0$, $\forall k \geq 0$, is strictly increasing. We also define $b_k = B_k - B_{k-1} > 0.$ By convention, the summation from $i$ to $j$ where $j < i$ is taken to be zero. Observe that
\begin{equation}\label{eq:FGM-C-0}
    \cc_0 = B_0 (f(\vx_0) - f(\vx^*)).
\end{equation}

While, in principle, one could also consider $\cc_k =  A_k \|\nabla f(\vx_k)\|^2 + B_k(f(\vx_k) - f(\vx^*))$ hoping to obtain a bound on the last gradient, it is not clear that such a bound is even possible for non-constant $B_k$ (see Section~\ref{sec:discussion}).

We first show that there is a natural algorithm that ensures $\cc_{k+1} - \cc_k \leq E_k,$ $\forall k \geq 0$, where $E_k$ contains only telescoping terms. As it turns out, this algorithm is precisely Nesterov FGM.
\begin{lemma}\label{lemma:fgm-ninc-c-k}
Given an arbitrary initial point $\vx_0 \in \rr^d$, assume that for $k \geq 1,$ the sequence $\vx_k$ is updated as
\begin{equation}\label{eq:Nesterov-FGM}
    \begin{gathered}
        \vx_{k} = \frac{B_{k-1}}{B_k}\Big(\vx_{k-1} - \frac{1}{L}\nabla f(\vx_{k-1})\Big) + \frac{b_k}{B_k}\vv_k,
    \end{gathered}
\end{equation}
where $\vv_k$ is defined recursively via $\vv_{k} = \vv_{k-1} - \frac{b_{k-1}}{L}\nabla f(\vx_{k-1})$ with $\vv_0 = \vx_0$. If ${b_k}^2 \leq B_k$ and $a_{k-1} \leq \frac{B_{k-1}}{2L},$ then  
 $\cc_{k} - \cc_{k-1} \leq \frac{L}{2}\big(\|\vx^* - \vv_k\|^2 - \|\vx^* - \vv_{k+1}\|^2\big),$ $\forall k \geq 1,$ where $\cc_k$ is defined by Eq.~\eqref{eq:C_k_FGM}.
\end{lemma}
\begin{proof}
Given $k \geq 1,$ by definition of $\cc_k,$ we have
\begin{equation}\label{eq:FGM-change-in-ck-1}
    \cc_k - \cc_{k-1} = a_{k-1} \|\nabla f(\vx_{k-1})\|^2 + B_k f(\vx_k) - B_{k-1}f(\vx_{k-1}) - b_k f(\vx^*). 
\end{equation}
Since $f(\vx^*)$ is not known to the algorithm and we are trying to bound $\cc_k - \cc_{k-1}$ above, it appears natural to use Eq.~\eqref{eq:smooth+cvx} to bound $f(\vx^*)$ below. In particular, we have:
\begin{equation}\label{eq:FGM-bnd-fx^*-below}
    f(\vx^*) \geq f(\vx_k) + \innp{\nabla f(\vx_k), \vx^* - \vx_k} + \frac{1}{2L}\|\nabla f(\vx_k)\|^2.  
\end{equation}
On the other hand, the difference $f(\vx_k) - f(\vx_{k-1})$ can be bounded above using, again, Eq.~\eqref{eq:smooth+cvx}, as follows.
\begin{equation}\label{eq:FGM-f-k-k-1}
\begin{aligned}
    f(\vx_{k}) - f(\vx_{k-1}) \leq\;& \innp{\nabla f(\vx_k), \vx_k - \vx_{k-1} + \frac{1}{L}\nabla f(\vx_{k-1})}\\
    &- \frac{1}{2L}\|\nabla f(\vx_k)\|^2 - \frac{1}{2L}\|\nabla f(\vx_{k-1})\|^2.
\end{aligned}
\end{equation}
Combining Eq.~\eqref{eq:FGM-bnd-fx^*-below} and Eq.~\eqref{eq:FGM-f-k-k-1} with Eq.~\eqref{eq:FGM-change-in-ck-1}, we have:
\begin{equation}\label{eq:FGM-change-in-ck-2}
    \begin{aligned}
        \cc_k - \cc_{k-1} \leq\;& - \frac{B_k}{2L}\|\nabla f(\vx_k)\|^2  + \Big(a_{k-1} - \frac{B_{k-1}}{2L}\Big)\|\nabla f(\vx_{k-1})\|^2\\ 
        &+ B_{k-1}\innp{\nabla f(\vx_k), \vx_k - \vx_{k-1} + \frac{1}{L}\nabla f(\vx_{k-1})}
        + b_k \innp{\nabla f(\vx_k), \vx_k - \vx^*}.
    \end{aligned}
\end{equation}
Now, if $b_k$ were zero (constant $B_k$), we could simply set $\vx_k = \vx_{k-1} - \frac{1}{L}\nabla f(\vx_{k-1}),$ and we would be recovering gradient descent and its analysis from the previous subsection. Of course, the goal here is to get a different trade off, where $B_k$ is strictly increasing. 

To get a useful bound on $\cc_k - \cc_{k-1},$ we need to be able to bound or otherwise control the term $b_k \innp{\nabla f(\vx_k), \vx_k - \vx^*}.$ Fortunately, such a term frequently appears in the mirror-descent-type analysis, and it can be bounded using standard arguments by defining
\begin{align*}
    \vv_{k+1} &= \argmin_{\vu \in \rr^d}\Big\{b_k \innp{\nabla f(\vx_k), \vu - \vv_k} + \frac{L}{2}\|\vu - \vv_k\|^2\Big\}\\
    & = \vv_k - \frac{b_k}{L}\nabla f(\vx_k).
\end{align*}
Then, we have:
\begin{align*}
    b_k \innp{\nabla f(\vx_k), \vx_k - \vx^*} =\; & b_k \innp{\nabla f(\vx_k), \vx_k - \vv_{k+1}} + L\innp{\vv_k - \vv_{k+1}, \vv_{k+1} - \vx^*}\\
    =\;& b_k\innp{\nabla f(\vx_k), \vx_k - \vv_k} + \frac{{b_k}^2}{L}\|\nabla f(\vx_k)\|^2 \\
    &+ \frac{L}{2}\|\vx^* - \vv_k\|^2 - \frac{L}{2}\|\vx^* - \vv_{k+1}\|^2 - \frac{L}{2}\|\vv_{k+1} - \vv_k\|^2\\
    =\;& b_k\innp{\nabla f(\vx_k), \vx_k - \vv_k} + \frac{{b_k}^2}{2L}\|\nabla f(\vx_k)\|^2\\
    &+ \frac{L}{2}\|\vx^* - \vv_k\|^2 - \frac{L}{2}\|\vx^* - \vv_{k+1}\|^2,
\end{align*}
where we have repeatedly used $\vv_{k+1} = \vv_k - \frac{b_k}{L}\nabla f(\vx_k).$ Combining with Eq.~\eqref{eq:FGM-change-in-ck-2}, we have
\begin{align*}
    \cc_k - \cc_{k-1} \leq\;& \frac{{b_k}^2 - B_k}{2L}\|\nabla f(\vx_k)\|^2  + \Big(a_{k-1} - \frac{B_{k-1}}{2L}\Big)\|\nabla f(\vx_{k-1})\|^2\\
    &+ \frac{L}{2}\|\vx^* - \vv_k\|^2 - \frac{L}{2}\|\vx^* - \vv_{k+1}\|^2\\ 
        &+ \innp{\nabla f(\vx_k), B_k\vx_k - B_{k-1}\Big(\vx_{k-1} - \frac{1}{L}\nabla f(\vx_{k-1})\Big) - b_k \vv_k}.
\end{align*}
To obtain $\cc_k - \cc_{k-1} \leq \frac{L}{2}\|\vx^* - \vv_k\|^2 - \frac{L}{2}\|\vx^* - \vv_{k+1}\|^2,$ it remains to choose ${b_k}^2 \leq B_k,$ $a_{k-1} \leq \frac{B_{k-1}}{2L}$, and $\vx_k = \frac{B_{k-1}}{B_k}\big(\vx_{k-1} - \frac{1}{L}\nabla f(\vx_{k-1})\big) + \frac{b_k}{B_k}\vv_k.$
\end{proof}

We can now use Lemma~\ref{lemma:fgm-ninc-c-k} to argue about convergence of Nesterov FGM from Eq.~\eqref{eq:Nesterov-FGM}. Interestingly, the result from Lemma~\ref{lemma:fgm-ninc-c-k} suffices to argue about both convergence in function value and in norm of the gradient. The resulting bounds are tight, up to a small absolute constant, due to numerical results from~\cite{kim2018generalizing}.
\begin{theorem}[Convergence of Fast Gradient Method]\label{thm:FGM-conv}
Suppose that the assumptions of Lemma~\ref{lemma:fgm-ninc-c-k} hold, where $\vv_0 = \vx_0.$ Then, $\forall k \geq 1$:
\begin{equation*}
    f(\vx_k) - f(\vx^*) \leq \frac{2B_0(f(\vx_0) - f(\vx^*)) + L\|\vx_0 - \vx^*\|^2}{2B_k} 
\end{equation*}
and
\begin{equation*}
    \sum_{i=0}^{k} \frac{B_i}{2L} \|\nabla f(\vx_i)\|^2 \leq  B_0(f(\vx_0) - f(\vx^*)) + \frac{L}{2}\|\vx_0 - \vx^*\|^2.
\end{equation*}
In particular, if $b_0 = B_0,$ ${b_k}^2 = B_k$  for $k \geq 1$, and $a_k = \frac{B_{k}}{2L}$, then
\begin{equation*}
    f(\vx_k) - f(\vx^*) \leq \frac{4 L\|\vx_0 - \vx^*\|^2}{(k+1)(k+2)}
\end{equation*}
and
\begin{equation*}
    \min_{0\leq i \leq k}\|\nabla f(\vx_i)\|^2 \leq \frac{\sum_{i=0}^{k} {B_i} \|\nabla f(\vx_i)\|^2}{\sum_{i=0}^k B_i} \leq  \frac{18 L^2 \|\vx_0 - \vx^*\|^2}{(k+1)(k+2)(k+3)}.
\end{equation*}
\end{theorem}
\begin{proof}
Applying Lemma~\ref{lemma:fgm-ninc-c-k} and the definition of $\cc_k,$ we have, $\forall k \geq 1$:
\begin{align*}
    \cc_k &\leq \cc_0 + \frac{L}{2}\|\vx^* - \vv_0\|^2 - \frac{L}{2}\|\vv_{k+1} - \vx^*\|^2\\
    &\leq B_0(f(\vx_0) - f(\vx^*)) + \frac{L}{2}\|\vx^* - \vx_0\|^2.
\end{align*}
Equivalently:
\begin{equation*}
    \sum_{i=0}^{k-1} a_i \|\nabla f(\vx_i)\|^2 + B_k(f(\vx_k) - f(\vx^*)) \leq B_0(f(\vx_0) - f(\vx^*)) + \frac{L}{2}\|\vx^* - \vx_0\|^2.
\end{equation*}
The first part of the theorem is now immediate, as $\sum_{i=0}^{k-1} a_i \|\nabla f(\vx_i)\|^2 \geq 0$ and $$
B_k(f(\vx_k) - f(\vx^*)) \geq \frac{B_k}{2L}\|\nabla f(\vx_k)\|^2 \geq a_k \|\nabla f(\vx_k)\|^2.
$$
For the second part, we only need to bound the growth of $B_k$ when ${b_k}^2 = (B_k - B_{k-1})^2 = B_k.$ It is a standard result that this growth is quadratic and at least as fast the growth resulting from choosing $b_k = \frac{k+1}{2},$ $\forall k.$ Thus, $B_k \geq \sum_{i=0}^k \frac{i+1}{2} = \frac{(k+1)(k+2)}{4}$ and $\sum_{i=0}^k B_i \geq \frac{(k+1)(k+2)(k+3)}{12}.$ Using that $f(\vx_0) - f(\vx^*) \leq \frac{L}{2}\|\vx_0 - \vx^*\|^2,$ it now follows from the first part of the theorem that 
$$
    f(\vx_k) - f(\vx^*) \leq \frac{4 L\|\vx_0 - \vx^*\|^2}{(k+1)(k+2)}
$$
and
$$
    \min_{0\leq i \leq k}\|\nabla f(\vx_i)\|^2 \leq \frac{\sum_{i=0}^{k} {B_i} \|\nabla f(\vx_i)\|^2}{\sum_{i=0}^k B_i} \leq \frac{18 L^2 \|\vx_0 - \vx^*\|^2}{(k+1)(k+2)(k+3)},   
$$
as claimed.
\end{proof}
\begin{remark}\label{rem:fgm}
It may not be immediately clear why the bound from Theorem~\ref{thm:FGM-conv} improves upon the bound for gradient descent from Lemma~\ref{lem:gd-cvx}, as in the former the gradient is bounded as a function of $\|\vx^* - \vx_0\|^2,$ while in the latter it is bounded as a function of $f(\vx_0)- f(\vx^*).$ Here, one should note that, using the standard convergence result for the optimality gap of gradient descent $f(\vx_k) - f(\vx^*) = O\big(\frac{L\|\vx_0 - \vx^*\|^2}{k}\big)$ and combining it with the bound from Lemma~\ref{lem:gd-cvx}, we also have that $\|\nabla f(\vx_k)\|^2 = O\big(L(\frac{f(\vx_{\lceil k/2 \rceil}) - f(\vx^*))}{k}\big) = O\big(\frac{L^2\|\vx_0 - \vx^*\|^2}{k^2}\big).$ Furthermore, this bound is known to be tight~\cite[Theorem 2]{kim2018generalizing}, and it also applies to $\min_{0\leq i \leq k}\|\nabla f(\vx_i)\|^2$, as gradient descent monotonically decreases the gradient. We also note that the improved bound for FGM from Theorem~\ref{thm:FGM-conv} can only be established for the minimum gradient norm up to iteration $k;$ as shown numerically in~\cite{kim2018generalizing}, the bound for the gradient of the last iterate is no better than that of gradient descent, i.e., $\|\nabla f(\vx_k)\|^2 = \Omega\big(\frac{L^2\|\vx_0 - \vx^*\|^2}{k^2}\big)$.
\end{remark}

\subsubsection{Optimized Method for the Gradients}

The only known method that achieves the optimal convergence bound of the form $\|\nabla f(\vx_k)\|^2 = O\big(\frac{L(f(\vx_0) - f(\vx^*))}{k^2}\big)$ is the optimized method for the gradients (OGM-G), due to Kim and Fessler~\cite{kim2020optimizing}. This method was obtained using the performance estimation framework (PEP) of Drori and Teboulle~\cite{drori2014performance}, which relies on numerical solutions to semidefinite programs that model the worst case performance of methods on a given class of problems (such as, e.g., unconstrained problems with smooth convex objective functions considered here). While this is a very powerful approach that generally produces tight convergence analysis and worst case instances as a byproduct, as discussed before, the intuition behind the methods and their analysis obtained using PEP is not always clear.

In this section, we show that OGM-G naturally arises from a potential function that fits within the broader framework studied in this paper. In particular, as mentioned earlier in this section, we can view OGM-G as trading off the norm of the gradient for a gap w.r.t.~an anchor point, which is the last point constructed by the algorithm. As a consequence of anchoring to the last point, the algorithm crucially requires fixing the number of iterations in advance to achieve the optimal convergence bound stated above.

The potential function used for analyzing OGM-G is defined by
\begin{equation}\label{eq:pot-fn-OGM-G}
    \cc_k = A_k\Big(\frac{1}{2L}\|\nabla f(\vx_k)\|^2 + \frac{1}{2L}\|\nabla f(\vx_K)\|^2 + f(\vx_k) - f(\vx_K)\Big),
\end{equation}
where $K$ is the total number of iterations for which OGM-G is invoked. 

Unlike for other algorithms, we will not be able to argue that $\cc_k - \cc_{k-1} \leq E_k$ for $E_k$ that is either zero or only contains telescoping terms. Instead, we will settle for a more modest goal of arguing that, under the appropriate choice of algorithm steps and growth of the sequence $A_k,$ we have
$
    \cc_K \leq \cc_0. 
$ 
Observe that, by the definition of $\cc_k$, if we can prove that $A_K/A_0 = \Omega(K^2),$ this condition immediately leads to the desired bound
$$
    \|\nabla f(\vx_K)\|^2 = O\Big(\frac{L(f(\vx_0)- f(\vx_K))}{K^2}\Big) = O\Big(\frac{L(f(\vx_0)- f(\vx^*))}{K^2}\Big). 
$$

As before, we define $a_k = A_k - A_{k-1}$ and assume it is strictly positive, for all $k$ (i.e., $A_k$ is strictly increasing). To bound $\cc_K,$ we start by bounding the change in the potential function $\cc_k - \cc_{k-1},$ for $k \geq 1,$ in the following lemma. Observe that the lemma itself is algorithm-independent.

\begin{lemma}\label{lemma:OGM-G-pot-change}
Let $\cc_k$ be defined by Eq.~\eqref{eq:pot-fn-OGM-G}, for all $k \in \{0, 1, \dots, K\}.$ Define $\vy_k = \vx_k - \frac{1}{L}\nabla f(\vx_k)$ for $k \geq 0,$ and set $\vy_{-1} = \vx_0.$ Then, $\forall 1\leq k \leq K:$
\begin{align*}
    \cc_k - \cc_{k-1} \leq\;& A_k \innp{\nabla f(\vx_k), \vx_k - \vy_{k-1}} - A_{k-1}\innp{\nabla f(\vx_{k-1}), \vx_{k-1} - \vy_{k-2}}\\
    &+ \innp{\nabla f(\vx_{k-1}), A_k \vy_{k-1} - A_{k-1}\vy_{k-2} - a_k \vy_K}.
\end{align*}
\end{lemma}
\begin{proof}
Let $\vx, \vxh$ be any two vectors from $\rr^d,$ and let $\vy = \vx - \frac{1}{L}\nabla f(\vx).$ Then, Eq.~\eqref{eq:smooth+cvx} can be equivalently written as:
\begin{equation}\label{eq:equiv-smooth+cvx}
    f(\vxh) - f(\vx) \leq \innp{\nabla f(\vxh), \vxh - \vy} - \frac{1}{2L}\|\nabla f(\vxh)\|^2 - \frac{1}{2L}\|\nabla f(\vx)\|^2.
\end{equation}

From the definition of $\cc_k$ in Eq.~\eqref{eq:pot-fn-OGM-G}, we have
\begin{align*}
    \cc_k - \cc_{k-1} =\;& \frac{A_k}{2L}\|\nabla f(\vx_k)\|^2 - \frac{A_{k-1}}{2L}\|\nabla f(\vx_{k-1})\|^2 + \frac{a_k}{2L}\|\nabla f(\vx_K)\|^2\\
    &+ A_k(f(\vx_k) - f(\vx_{k-1})) + a_k(f(\vx_{k-1}) - f(\vx_K)).
\end{align*}
Applying Eq.~\eqref{eq:equiv-smooth+cvx} to $f(\vx_k) - f(\vx_{k-1})$ and $f(\vx_{k-1}) - f(\vx_K)$, we further have
\begin{align*}
    \cc_k - \cc_{k-1} \leq\;& - \frac{A_k}{L}\|\nabla f(\vx_{k-1})\|^2 + A_k \innp{\nabla f(\vx_k), \vx_k - \vy_{k-1}}\\
    &+ a_k \innp{\nabla f(\vx_{k-1}), \vx_{k-1} - \vy_K}\\
    =\;& A_k \innp{\nabla f(\vx_k), \vx_k - \vy_{k-1}} + A_{k}\innp{\nabla f(\vx_{k-1}), \vy_{k-1} - \vy_K}\\
    &- A_{k-1}\innp{\nabla f(\vx_{k-1}), \vx_{k-1} - \vy_K}\\
    =\;& A_k \innp{\nabla f(\vx_k), \vx_k - \vy_{k-1}} - A_{k-1}\innp{\nabla f(\vx_{k-1}), \vx_{k-1} - \vy_{k-2}} \\
    &+ \innp{\nabla f(\vx_{k-1}), A_k \vy_{k-1}, - A_{k-1}\vy_{k-2} - a_k \vy_K},
\end{align*}
as claimed.
\end{proof}

The following lemma provides the restrictions on the step sizes of the algorithm that are needed to ensure that $\cc_K \leq \cc_0.$ Here, we assume that each point $\vx_k$ can be expressed as the sum of the initial point $\vx_0$ and some linear combination of the gradients evaluated at points $\vx_i$ for $0\leq i \leq k-1.$ Note that most of the standard first-order algorithms can be expressed in this form.

\begin{lemma}\label{lemma:OGM-G-final-pot}
Let $\cc_k$ be defined by Eq.~\eqref{eq:pot-fn-OGM-G} for $k \in \{0,\dots, K\}$ and assume that points $\vx_k$ can be expressed as $\vx_k = \vx_0 - \frac{1}{L}\sum_{i=0}^{k-1}\beta_{i, k}\nabla f(\vx_i),$ where $\beta_{i, k}$ are some real scalars. Define $\beta_{k, k} = 1,$ so that $\vy_k = \vx_k - \frac{1}{L}\nabla f(\vx_k) = \vx_0 - \frac{1}{L}\sum_{i=0}^k \beta_{i, k}\nabla f(\vx_i)$ and set $\vy_{-1} = \vx_0.$ If the following two conditions are satisfied for all $0\leq j < k \leq K-1$:
\begin{gather}
    \beta_{k, K-1} + \frac{a_{k+1}}{A_K} \leq \frac{A_{k+1}}{a_{k+1}}, \label{eq:OGM-G-cond-1}\\
    A_{k+1}\beta_{j, k} = A_k \beta_{j, k-1} + a_{k+1}\Big(\beta_{j, K-1} + \frac{a_{j+1}}{A_K}\Big) + a_{j+1}\Big(\beta_{k, K-1} + \frac{a_{k+1}}{A_K}\Big) \label{eq:OGM-G-cond-2}
\end{gather}
and if
\begin{equation}\label{eq:OGM-G-x_K}
    \vx_K = \vy_{K-1} - \frac{1}{L A_K}\sum_{k=0}^{K-1}a_{k+1}\nabla f(\vx_k),
\end{equation}
then $\cc_K \leq \cc_0.$ Further, the largest growth of $\frac{A_K}{A_0}$ for which both of these conditions can be satisfied is $O(K^2).$ 
\end{lemma}
\begin{proof}
Telescoping the inequality from Lemma~\ref{lemma:OGM-G-pot-change}, we have:
\begin{align*}
    \cc_K - \cc_0 \leq\;& A_K \innp{\nabla f(\vx_K), \vx_K - \vy_{K-1}}\\
    &+ \sum_{k=0}^{K-1}\innp{\nabla f(\vx_k), A_{k+1}\vy_k - A_k \vy_{k-1} - a_{k+1}\vy_K}.
\end{align*}
Observe that $\nabla f(\vx_K)$ only appears in the first term and as part of $\vy_K = \vx_K - \frac{1}{L}\nabla f(\vx_K).$ Thus, grouping the terms that multiply $\nabla f(\vx_K),$ we can equivalently write
\begin{align*}
    \cc_K - \cc_0 \leq\;& \innp{\nabla f(\vx_K), A_K(\vx_K - \vy_{K-1}) + \frac{1}{L}\sum_{k=0}^{K-1}a_{k+1}\nabla f(\vx_k)}\\
    &+ \sum_{k=0}^{K-1}\innp{\nabla f(\vx_k), A_{k+1}\vy_k - A_k \vy_{k-1} - a_{k+1}\vx_K}. 
\end{align*}
The choice of $\vx_K$ from Eq.~\eqref{eq:OGM-G-x_K} ensures that the first term on the right-hand side is zero (and this is how it was chosen). The rest of the terms can be expressed as a function of gradients up to the $(K-1)^{\mathrm{th}}$ one. To simplify the notation, let us define $\vg_{K-1} = \frac{1}{L}\sum_{k=0}^{K-1}a_{k+1}\nabla f(\vx_k).$ Then, we have
\begin{align}\label{eq:OGM-G-C_K-C_0-final}
    \cc_K - \cc_0 \leq \sum_{k=0}^{K-1}\innp{\nabla f(\vx_k), A_{k+1}\vy_k - A_k \vy_{k-1} - a_{k+1}\Big(\vy_{K-1} - \frac{\vg_{K-1}}{A_K}\Big)}.
\end{align}
Observe that, as $\vy_k = \vx_0 - \frac{1}{L}\sum_{i=0}^k \beta_{i, k}\nabla f(\vx_i)$ by the lemma assumptions, the expression on the right-hand side can be written as a linear combination of inner products between gradients, as follows.
\begin{align*}
    \cc_K - \cc_0 \leq \frac{1}{L}\sum_{j=0}^{K-1}\sum_{k=j}^{K-1}P_{j, k}\innp{\nabla f(\vx_j), {\nabla f(\vx_k)}},
\end{align*}
where, by Eq.~\eqref{eq:OGM-G-C_K-C_0-final}, we have that, for all $0 \leq j < k \leq K-1:$
\begin{align*}
    P_{k,k} &= -A_{k+1}\beta_{k, k} + a_{k+1}\Big(\beta_{k, K-1} + \frac{a_{k+1}}{{A_K}}\Big),\\
    P_{j,k} &= - A_{k+1}\beta_{j, k} + A_k \beta_{j, k-1} + a_{k+1}\Big(\beta_{j, K-1} + \frac{a_{j+1}}{A_K}\Big) + a_{j+1}\Big(\beta_{k, K-1} + \frac{a_{k+1}}{A_K}\Big).
\end{align*}
As, by assumption, $\beta_{k,k} = 1,$ conditions in Eqs.~\eqref{eq:OGM-G-cond-1} and \eqref{eq:OGM-G-cond-2} are equivalent to $P_{k, k} \leq 0$ and $P_{j, k} = 0$, for all $0 \leq j < k \leq K-1.$ By construction, these conditions are sufficient for guaranteeing $\cc_k - \cc_0 \leq 0,$ completing the first part of the proof.

Observe that, given a sequence of positive numbers $\{a_k\}_{k\geq 0}$ and $A_k = \sum_{j=0}^k a_j,$ all coefficients $\beta_{j, k}$ are uniquely determined by Eq.~\eqref{eq:OGM-G-cond-2} (as $\beta_{k, k} = 1$ by assumption, and the remaining coefficients can be computed by recursively applying Eq.~\eqref{eq:OGM-G-cond-2}). Thus, the role of the condition from Eq.~\eqref{eq:OGM-G-cond-1} is to limit the growth of the sequence $\{A_k\}_{k\geq 0}.$ Starting with $\beta_{k, k} = 1,$ $\forall k$ (which holds by assumption), it is possible to argue by induction that $\beta_{j, k} \geq 0,$ $\forall j, k$ (the proof is omitted for brevity). Thus the condition  from Eq.~\eqref{eq:OGM-G-cond-1} implies that $\frac{a_{k+1}}{A_{K}}\leq \frac{A_{k+1}}{a_{k+1}}.$ Equivalently, $\forall k \leq K-1$:
\begin{equation}\label{eq:OGM_max-growth}
    \frac{{a_{k+1}}^2}{A_{k+1}} \leq A_K. 
\end{equation}
For any fixed $A_K,$ Eq.~\eqref{eq:OGM_max-growth} implies that $\frac{A_k}{A_0}$ cannot grow faster than quadratically with $k,$ for $k \leq K-1.$ It remains to argue that the sequence does not make a big jump from $A_{K-1}$ to $A_K.$ This follows by using again Eq.~\eqref{eq:OGM-G-cond-1} for $k = K-1$ and recalling that $\beta_{K-1, K-1} = 1.$ We then have
$$
    1 + \frac{a_K}{A_K} \leq \frac{A_K}{a_K}.
$$
Solving for $\frac{a_K}{A_K},$ it follows that $\frac{a_K}{A_K} \leq \frac{-1 + \sqrt{5}}{2} < 0.62,$ and, thus, $\frac{A_K}{A_{K-1}}\leq \frac{1}{1-0.62} < 3,$ completing the proof that $\frac{A_K}{A_0} = O(K^2).$
\end{proof}
That $\frac{A_K}{A_0} = O(K^2)$ is not surprising -- if it were not true, by the discussion from the beginning of this subsection, we would be able to obtain an algorithm that converges at rate faster than $1/K^2,$ which is impossible, due to the existing lower bounds~\cite{carmon2019lower}. This result was rather included to highlight the role of the conditions from Eqs.~\eqref{eq:OGM-G-cond-1} and \eqref{eq:OGM-G-cond-2} in Lemma~\ref{lemma:OGM-G-final-pot}: the first condition limits the growth of $\{A_k\}_{k\geq 0},$ whereas the second determines the step sizes $\beta_{j, k}$ in the algorithm, given the sequence $\{A_k\}_{k\geq 0}$. 

What remains to be shown is that there is a choice of step sizes $\beta_{j, k}$ that guarantees $\frac{A_K}{A_0} = \Theta(K^2),$ and thus leads to an algorithm with the optimal convergence rate. This choice is obtained when the inequality from Eq.~\eqref{eq:OGM-G-cond-1} is satisfied with equality. It is possible to argue that this choice is also the one that leads to the fastest growth of $\frac{A_K}{A_0};$ however, this direction is not pursued here as it unnecessarily complicates the analysis. Further, when Eq.~\eqref{eq:OGM-G-cond-1} is satisfied with equality, Eq.~\eqref{eq:OGM-G-cond-2} can be further simplified, and it leads to the algorithm description that does not necessitate storing all of the gradients, but only a constant number of $d$-dimensional vectors. However, similar to the algorithm description in~\cite{kim2020optimizing}, the entire sequence $\{A_k\}_{k=0}^K$ needs to be pre-computed and stored, which appears to be unavoidable. 
The algorithm and its convergence rate are summarized in the following theorem.

\begin{theorem}[Convergence of Optimized Gradient Method]\label{thm:OGM-G}
Let $f: \rr^d \to \rr$ be an $L$-smooth function and let $\vx_0 \in \rr^d$ be an arbitrary initial point. Let $K \geq 1.$ Consider the following algorithm. Let $\vv = \vx_0 - \frac{A_1}{a_1 L}\nabla f(\vx_0)$, $\vg_0 = a_1 \nabla f(\vx_0).$ For $k = 1$ to $K-1,$ 
\begin{equation}\label{eq:OGM-G}
    \begin{gathered}
        \vy_{k-1} = \vx_{k-1} - \frac{1}{L}\nabla f(\vx_{k-1}),\\
        \vx_k = \frac{A_k}{A_{k+1}}\vy_{k-1} + \frac{a_{k+1}}{A_{k+1}}\vv_{k-1} - \frac{1}{a_{k+1}}\vg_{k-1},\\
        \vv_k = \vv_{k-1} - \frac{1}{L}\frac{A_{k+1}}{a_{k+1}}\nabla f(\vx_k), \; \vg_k = \vg_{k-1} + a_{k+1}\nabla f(\vx_k),
    \end{gathered}
\end{equation}
where the sequence $\{A_k\}_{k=0}^K$ is recursively defined by the following 
\begin{equation}\label{eq:A_k expression}
    \begin{cases}
    A_k = 1, & \text{ if } k = K,\\
    A_{k} = A_{k+1}\big[1+\frac{1}{2}A_{k+1}-\frac{1}{2}\sqrt{A_{k+1}(4+A_{k+1})}\big], & \text{ if } 0 \leq k \leq K-1,
    \end{cases}
\end{equation}
and $a_{k+1} = A_{k+1} - A_k,$ for $0 \leq k \leq K-1.$ 

If $\vx_K$ is defined by
$$
    \vx_K = \vy_{K-1} - \frac{1}{A_K L}\vg_{K-1},
$$
then
$$
    \|\nabla f(\vx_K)\|^2 \leq  \frac{16 L(f(\vx_0) - f(\vx^*))}{(K+2)^2},
$$
where $\vx^* \in \argmin_{\vx \in \rr^d} f(\vx).$
\end{theorem}
\begin{proof}
The proof strategy is as follows. We first argue that the algorithm from the theorem statement satisfies $\cc_K \leq \cc_0,$ where $\cc_k$ is defined by Eq.~\eqref{eq:pot-fn-OGM-G}. This is done by showing that we can apply Lemma~\ref{lemma:OGM-G-final-pot}. Then, by the definition of $\cc_k,$ $\cc_K \leq \cc_0$ is equivalent to
$$
    \|\nabla f(\vx_K)\|^2 \leq 2L\frac{A_0}{A_K}\Big(f(\vx_0) - f(\vx_K) + \frac{1}{2L}\|\nabla f(\vx_0)\|^2\Big).
$$
As $f(\vx_K) \geq f(\vx^*)$ and $\frac{1}{2L}\|\nabla f(\vx_0)\|^2 \leq f(\vx_0) - f(\vx^*),$ what then remains to be argued is that $\frac{A_0}{A_K} = O(\frac{1}{K^2}).$ 

To apply Lemma~\ref{lemma:OGM-G-final-pot}, observe first that the definition of $\vx_K$ from the theorem statement is the same as the definition of $\vx_K$ in Lemma~\ref{lemma:OGM-G-final-pot}. For $k \leq K-1,$ let us define $\vx_k = \vx_0 - \frac{1}{L}\sum_{j=0}^{k-1} \beta_{j, k}\nabla f(\vx_j),$ $\beta_{k, k} = 1,$ and $\vy_k = \vx_k - \frac{\beta_{k,k}}{L}\nabla f(\vx_k)$ as in Lemma~\ref{lemma:OGM-G-final-pot} and show that when both conditions from Lemma~\ref{lemma:OGM-G-final-pot} stated in Eqs.~\eqref{eq:OGM-G-cond-1} and Eq.~\eqref{eq:OGM-G-cond-2} are satisfied with equality, we recover the algorithm from the theorem statement, and thus the two sequences of points are equivalent, and so we can conclude that $\cc_K \leq \cc_0.$ 

When Eq.~\eqref{eq:OGM-G-cond-1} holds with equality, we have that \begin{equation}\label{eq:b_k,K-1}
    \beta_{k, K-1} + \frac{a_{k+1}}{A_K} = \frac{A_{k+1}}{a_{k+1}}.
\end{equation} 
Plugging it into Eq.~\eqref{eq:OGM-G-cond-2}, we have
\begin{equation}\label{eq:A_k,b_j,k}
    A_{k+1}\beta_{j, k} = A_k \beta_{j, k-1} + a_{k+1} \frac{A_{j+1}}{a_{j+1}} + a_{j+1} \frac{A_{k+1}}{a_{k+1}}.
\end{equation}
Thus, it follows that
\begin{align*}
    A_{k+1}\vx_k - A_k \vy_{k-1} &= a_{k+1} \vx_0 - \frac{a_{k+1}}{L} \sum_{j=0}^{k-1}\frac{A_{j+1}}{a_{j+1}} \nabla f(\vx_j) - \frac{a_{j+1}}{L}\sum_{j=0}^{k-1}a_{j+1}\nabla f(\vx_j)\\
    &= a_{k+1}\vv_{k-1} - \frac{A_{k+1}}{a_{k+1}}\vg_{k-1},
\end{align*}
which is the same as the definition of $\vx_k$ from Eq.~\eqref{eq:OGM-G}. 

It remains to show that the conditions from Lemma~\ref{lemma:OGM-G-final-pot} imply the recursive definition of the sequence $\{A_k\}_{k\geq 0}$ and that $\frac{A_K}{A_0} \geq \frac{4}{(K+2)^2}.$ This is established by Lemma~\ref{lemma:OGM-G-A_k-growth} in the appendix. 
\end{proof}
\begin{remark}\label{rem:OGM-G}
While OGM-G provides the optimal convergence guarantee for norm of the gradient, its convergence rate for the optimality gap is not known. Thus, it does not immediately imply a bound on norm of the gradient in terms of $\|\vx^* - \vx_0\|^2.$ However, as observed in~\cite{nesterov2020primal}, it is possible to obtain a bound of $\|\nabla f(\vx_K)\|^2 = O\big(\frac{L^2\|\vx^* - \vx_0\|^2}{K^4}\big)$ from OGM-G, by running Nesterov FGM for $\lfloor K/2 \rfloor$ iterations, followed by  $\lceil K/2 \rceil$ iterations of OGM-G.
\end{remark}
%
\subsection{Discussion} \label{sec:discussion}

Gradient descent is perhaps the simplest method that can be used for minimizing the gradient norm. We also conjecture that it is, in a certain sense, optimal. 
\begin{conjecture}\label{conj:gd-optimality}
For any $K > 0$ and any method that constructs its iterates as $\vx_k = \vx_0 - \sum_{i=0}^{k-1}\beta_{i, k} \nabla f(\vx_i),$ where $\vx_0 \in \rr^d$ is the initial point, $f$ is a convex function accessed via a gradient oracle, and coefficients $\beta_{i, k} \in \rr$ can depend on $L > 0, i, k$ but are otherwise chosen independently of $K$ or the input function $f,$ there exists an $L$-smooth convex input function $f$ and an absolute constant $C > 0$ such that 
$$
    \|\nabla f(\vx_K)\|^2 \geq C\frac{L(f(\vx_0) - f(\vx^*))}{K}.
$$
\end{conjecture}
The basis for this conjecture is the numerical evidence from~\cite{kim2018generalizing, kim2020optimizing}, which seems to suggest that fixing the total number of iterations $K$ and choosing the coefficients $\beta_{i, k}$ as a function $K$ is crucial to obtaining the optimal bound $\|\nabla f(\vx_K)\|^2 = O\Big(\frac{L(f(\vx_0) - f(\vx^*))}{K^2}\Big)$. We note that the lower bound from Conjecture~\ref{conj:gd-optimality} can be proved under a stricter condition on coefficients $\beta_{i, k}$ that essentially forces them to be constant (independent of $i$ and $k$), using the techniques of Arjevani and Shamir~\cite{arjevani2016iteration}. However, such a lower bound is weak as it not only excludes the optimal algorithm from~\cite{kim2020optimizing} (which is desired) but also all variants of Nesterov FGM considered in~\cite{kim2018generalizing}. 


\section{Small Gradients in Min-Max Optimization}\label{sec:min-max-opt}

In this section, we consider the problem of making the gradients small in convex-concave min-max optimization, under the assumption that the operator $F$ corresponding to the gradient of the objective is cocoercive (see Section~\ref{sec:prelims}). Similarly as in the case of convex optimization, the potential functions we consider trade off a notion of an optimality gap with the norm of $F.$ Further, the inequality corresponding to the cocoercivity assumption suffices to carry out the analysis of standard methods considered here; namely, the gradient descent-ascent method and Halpern iteration. We also show (in Section~\ref{sec:min-max-lb}) that these two methods are the best we can hope for when considering broad classes of methods that capture most of the standard optimization methods.   
%
\subsection{Krasnosel'ski{\i}-Mann/Gradient Descent-Ascent}\label{sec:gda}
 Perhaps 
the simp\-lest potential function that can be considered for min-max optimization is 
\begin{equation}\label{eq:gda-pot-fun}
    \cc_k = A_k \|F(\vu_k)\|^2 + B_k \innp{F(\vu_k), \vu_k - \vu^*},
\end{equation}
which can be seen as a counterpart to the potential function used for gradient descent in the previous section. The method that is suitable for the analysis with this potential function is also the counterpart of gradient descent for min-max optimization---gradient descent-ascent (GDA), stated as
$$
    \vu_{k+1} = \vu_k - \eta_k F(\vu_k),
$$
where $\eta_k \in (0, \frac{2}{L}).$ 
This method is also equivalent to the well-known Krasnosel'ski{\i}-Mann iteration for finding fixed points of nonexpansive (1-Lipschitz) operators. In particular, given a nonexpansive operator $T:\rr^d \to\rr^d,$ the Krasnosel'ski{\i}-Mann iteration updates the iterates as
$$
    \vu_{k+1} = (1-\alpha_k)\vu_k + \alpha_k T(\vu_k), 
$$
where $\alpha_k \in (0, 1).$ 
It is a standard fact that $F$ is $\frac{1}{L}$-cocoercive if and only if $T(\cdot) = \cdot - \frac{2}{L}F(\cdot)$ is nonexpansive (see, e.g.,~\cite[Proposition 4.1]{bauschke2011convex}). Thus, if we apply the Krasnosel'ski{\i}-Mann iteration to $T(\cdot) = \cdot - \frac{2}{L}F(\cdot)$, we have
$$
    \vu_{k+1} = \vu_k - \frac{2\alpha_k}{L}F(\vu_k),
$$
which is precisely GDA with $\eta_k = \frac{2\alpha_k}{L}.$ 

For simplicity, in the following we analyze GDA with the step size $\eta_k = \eta = \frac{1}{L},$ which is the optimal step size for this method. The analysis however extends to any step sizes $\eta_k \in (0, \frac{2}{L})$ in a straightforward manner. The convergence result is summarized in the following lemma.

\begin{lemma}[Convergence of Gradient Descent-Ascent]
Let $F:\rr^d\to\rr^d$ be a $\frac{1}{L}$-cocoercive operator, $\vu_0 \in \rr^d$ be an arbitrary initial point, and let $\vu_{k+1} = \vu_k - \frac{1}{L}F(\vu_k)$ for $k\geq 0.$ Then, $\forall k \geq 1:$
$$
    \|F(\vu_k)\| \leq \frac{L\|\vu_0 - \vu^*\|}{\sqrt{k/2 + 1}},
$$
where $\vu^*$ is such that $F(\vu^*) = \zeros.$
\end{lemma}
\begin{proof}
The proof relies on showing that the potential function $\cc_k$ satisfies $\cc_k \leq \cc_{k-1} + E_k,$ where $E_k$ only contains terms that telescope, for suitably chosen sequences of positive numbers $\{A_k\}_{k \geq 0}$ and $\{B_k\}_{k\geq 0}.$ 

Let us start with bounding $\cc_0.$ As $\vu_1 = \vu_0 - \frac{1}{L}F(\vu_0),$ we have
\begin{align*}
    \cc_0 &= A_0 \|F(\vu_0)\|^2 + B_0 \innp{F(\vu_0), \vu_0 - \vu^*}\\
    &= A_0 \|F(\vu_0)\|^2 + B_0 L \innp{\vu_0 - \vu_1, \vu_0 - \vu^*}\\
    &= A_0 \|F(\vu_0)\|^2 + \frac{B_0 L}{2}\big(\|\vu_0 - \vu^*\|^2 - \|\vu_1 - \vu^*\|^2 + \|\vu_0 - \vu_1\|^2\big)\\
    &= \Big(A_0 + \frac{B_0}{2L}\Big)\|F(\vu_0)\|^2 + \frac{B_0 L}{2}\big(\|\vu_0 - \vu^*\|^2 - \|\vu_1 - \vu^*\|^2\big).
\end{align*}
Eq.~\eqref{eq:minty-gap} implies $\|F(\vu_0)\|^2 \leq L^2\|\vu_0 - \vu^*\|^2,$ and, thus, we have
\begin{equation}\label{eq:gda-init-pot}
    \cc_0 = \frac{A_0 L^2 + 2 B_0 L}{2}\|\vu_0 - \vu^*\|^2 - \frac{B_0 L}{2}\|\vu_1 - \vu^*\|^2.
\end{equation}

Now let us consider the change in the potential function $\cc_k - \cc_{k-1}.$ Note first that, by Eq.~\eqref{eq:minty-gap}, $\innp{F(\vu_{k-1}), \vu_{k-1} - \vu^*} \geq \frac{1}{L}\|F(\vu_{k-1})\|^2.$ Thus:
\begin{align*}
    \cc_k - \cc_{k-1} =\;& A_k \|F(\vu_k)\|^2 - A_{k-1}\|F(\vu_{k-1})\|^2 + B_k \innp{F(\vu_k), \vu_k - \vu^*}\\
    &- B_{k-1}\innp{F(\vu_{k-1}), \vu_{k-1} - \vu^*} \\
    \leq \; & A_k \|F(\vu_k)\|^2-\Big(A_{k-1} + \frac{B_{k-1}}{L}\Big)\|F(\vu_{k-1})\|^2 + B_k \innp{F(\vu_k), \vu_k - \vu^*}.
\end{align*}
Using that $F(\vu_k) = L(\vu_k - \vu_{k+1}),$ we have that $\innp{F(\vu_k), \vu_k - \vu^*} = \frac{1}{2L}\|F(\vu_k)\|^2 + \frac{L}{2}\|\vu_k - \vu^*\|^2 - \frac{L}{2}\|\vu_{k+1} - \vu^*\|^2,$ which leads to
\begin{align*}
    \cc_k - \cc_{k-1} \leq \; & \Big(A_k + \frac{B_k}{2L}\Big)\|F(\vu_k)\|^2 - \Big(A_{k-1} + \frac{B_{k-1}}{L}\Big)\|F(\vu_{k-1})\|^2\\
    &+ \frac{B_k L}{2}\|\vu_k - \vu^*\|^2 - \frac{B_k L}{2}\|\vu_{k+1} - \vu^*\|^2. 
\end{align*}
On the other hand, by Eq.~\eqref{eq:cocoercive-op} and $\vu_k = \vu_{k-1} - \frac{1}{L}F(\vu_{k-1})$, we have that $\|F(\vu_k)\|^2 \leq \innp{F(\vu_k), F(\vu_{k-1})},$ and, consequently, $\|F(\vu_k)\| \leq \|F(\vu_{k-1})\|.$ Thus, for $\cc_k - \cc_{k-1}$ to contain only telescoping terms, it suffices that $A_k + \frac{B_k}{2 L} - A_{k-1} - \frac{B_{k-1}}{L} \leq 0$ and that $\{B_k\}_{k \geq 0}$ is non-increasing. In particular, taking $B_k = 1$ and $A_{k+1} = A_{k} + \frac{1}{2L} = A_0 + \frac{k+1}{2L}$ for all $k \geq 0,$ we have
\begin{equation}\label{eq:gda-telescoping}
    \cc_k - \cc_{k-1} \leq \frac{L}{2}\|\vu_k - \vu^*\|^2 - \frac{L}{2}\|\vu_{k+1} - \vu^*\|^2.
\end{equation}
Telescoping Eq.~\eqref{eq:gda-telescoping} and combining with Eq.~\eqref{eq:gda-init-pot}, we then get
$$
    \cc_k \leq \frac{A_0 L^2 + 2 L}{2}\|\vu_0 - \vu^*\|^2 - \frac{L}{2}\|\vu_{k+1} - \vu^*\|^2 \leq \frac{A_0 L^2 + 2 L}{2}\|\vu_0 - \vu^*\|^2.
$$
Taking $A_0 = 0$ and observing that, by Eq.~\eqref{eq:minty-gap}, $\cc_k \geq \big(A_k + \frac{B_k}{L}\big)\|F(\vu_k)\|^2 = \frac{k + 2}{2L} \|F(\vu_k)\|^2,$ we finally get
$$
    \|F(\vu_k)\|^2 \leq \frac{2L^2\|\vu_0 - \vu^*\|^2}{k + 2}.
$$
It remains to take the square-root on both sides of the last inequality. 
\end{proof}
 %
\subsection{Halpern Iteration}\label{sec:halpern}
It seems reasonable now to ask whether it is possible to obtain faster rates than for GDA by considering a different potential function that trades off the gradient/operator norm for a notion of an optimality gap w.r.t.~an anchor point, similar to how we obtained faster rates for convex optimization. It turns out that the answer is ``yes,'' using the initial point $\vu_0$ as the anchor. The resulting potential function is 
$$\cc_k = A_k \|F(\vu_k)\|^2 + B_k \innp{F(\vu_k), \vu_k - \vu_0}$$
and it corresponds to the well-known Halpern iteration
\begin{equation}\label{eq:halpern-iteration}
    \vu_{k+1} = \lambda_{k+1} \vu_0 + (1-\lambda_{k+1})T(\vu_k),
\end{equation}
where, similarly as in the case of GDA, $T(\cdot) = \cdot - \frac{2}{L}F(\cdot)$ is a nonexpansive operator. We note that a similar potential function was used in~\cite{diakonikolas2020halpern} to analyze the convergence of Halpern iteration.

Here we show that the potential function $\cc_k = A_k \|F(\vu_k)\|^2 + B_k \innp{F(\vu_k), \vu_k - \vu_0}$ in fact leads to the Halpern iteration as a natural algorithm that guarantees that $\cc_k$ is non-increasing. The main convergence result is summarized in the following lemma, whose proof reveals how the chosen potential function leads to the Halpern iteration. 

\begin{lemma}[Convergence of Halpern Iteration]\label{lemma:halpern}
Let $F:\rr^d\to\rr^d$ be a $\frac{1}{L}$-cocoercive operator, $\vu_0 \in \rr^d$ be an arbitrary initial point, and, for $k \geq 0,$ let
$$
    \vu_{k+1} = \frac{1}{k+1}\vu_0 + \frac{k}{k+1}\Big(\vu_k - \frac{2}{L}F(\vu_k)\Big).
$$
Then, $\forall k \geq 1,$ we have
$$
    \|F(\vu_k)\| \leq \frac{L\|\vu_0 - \vu^*\|}{k+1},
$$
where $\vu^*$ satisfies $F(\vu^*) = \zeros.$
\end{lemma}
\begin{proof}
The claim trivially holds if $\|F(\vu_k)\| = 0,$ so assume throughout that $\|F(\vu_k)\|\neq 0.$

Consider bounding $\cc_k - \cc_{k-1}$ above by zero. To do so, we can only rely on cocoercivity of $F$ from Eq.~\eqref{eq:cocoercive-op}. Applying Eq.~\eqref{eq:cocoercive-op} with $\vu = \vu_k$ and $\vv = \vu_{k-1}$ and rearranging the terms, we have
\begin{equation}\label{eq:halpern-equiv-cocoercive}
\begin{aligned}
    \frac{1}{L}\|F(\vu_k)\|^2 \leq\;& \innp{F(\vu_k), \vu_k - \vu_{k-1} + \frac{2}{L}F(\vu_{k-1})}\\
    &- \innp{F(\vu_{k-1}), \vu_{k} - \vu_{k-1}} - \frac{1}{L}\|F(\vu_{k-1})\|^2. 
\end{aligned}
\end{equation}
Combining Eq.~\eqref{eq:halpern-equiv-cocoercive} with the definition of $\cc_k$ and grouping appropriate terms, we have
\begin{equation}\label{eq:halpern-pot-change}
\begin{aligned}
    \cc_k - \cc_{k-1} \leq \;& \innp{F(\vu_k), A_k L \Big(\vu_k - \vu_{k-1} + \frac{2}{L}F(\vu_{k-1})\Big) + B_k(\vu_k - \vu_0)}\\
    &- \innp{F(\vu_{k-1}), A_k L (\vu_{k} - \vu_{k-1}) + B_{k-1}(\vu_{k-1} - \vu_{0})}\\
    &- \frac{A_k + A_{k-1}}{L}\|F(\vu_{k-1})\|^2. 
\end{aligned}
\end{equation}
For $\vu_k$ to be explicitly defined, it cannot depend on $F(\vu_k)$. Thus, the only direct way to make the first line of the right-hand side of Eq.~\eqref{eq:halpern-pot-change} non-positive is to set 
\begin{equation}\label{eq:halpern-cond-1}
    A_k L \Big(\vu_k - \vu_{k-1} + \frac{2}{L}F(\vu_{k-1}))\Big) + B_k(\vu_k - \vu_0) = 0.
\end{equation}
For the remaining terms, it suffices that
\begin{equation}\label{eq:halpern-cond-2} 
    - \innp{F(\vu_{k-1}), A_k L (\vu_{k} - \vu_{k-1}) + B_{k-1}(\vu_{k-1} - \vu_{0})} - ({A_k + A_{k-1}})\|F(\vu_{k-1})\|^2 \leq 0.
\end{equation}
Rearranging Eq.~\eqref{eq:halpern-cond-1} gives the Halpern algorithm from Eq.~\eqref{eq:halpern-iteration} with $\lambda_{k} = \frac{B_k}{A_k L + B_k},$ i.e.,
\begin{equation}\label{eq:halpern-it-equiv}
    \vu_k = \frac{B_k}{A_k L + B_k}\vu_0 + \frac{A_k L }{A_k L + B_k}\Big(\vu_{k-1} - \frac{2}{L}F(\vu_{k-1})\Big).
\end{equation}
The other condition (from Eq.~\eqref{eq:halpern-cond-2}) effectively constrains the growth of $A_k$ compared to $B_k,$ which is expected, as otherwise we would be able to prove an arbitrarily fast convergence rate for Halpern iteration, which is impossible, due to existing lower bounds (see, e.g.~\cite{diakonikolas2020halpern}).

Combining Eq.~\eqref{eq:halpern-cond-1} and Eq.~\eqref{eq:halpern-cond-2}, we have
\begin{align*}
    -\innp{F(\vu_{k-1}), A_k L \vu_k - B_{k-1}\vu_0 - (A_k L - B_{k-1})\vu_{k-1}} \leq (A_k + A_{k-1})\|F(\vu_{k-1})\|^2. 
\end{align*}
Now, to be able to guarantee that the last inequality is satisfied and consistent with Eq.~\eqref{eq:halpern-it-equiv}, it is required that 
\begin{equation}\label{eq:halpern-growth-condition}
    \frac{B_{k-1}}{A_k L} = \frac{B_k}{A_k L + B_k}\quad \text{ and } \quad \frac{2 A_k}{A_k L + B_k} \leq \frac{A_{k} + A_{k-1}}{A_k L}.
\end{equation}
In particular, when $B_k = k+1$ and $A_k = \frac{k(k+1)}{L},$ both conditions from Eq.~\eqref{eq:halpern-growth-condition} are satisfied with equality.

Hence, for $B_k = k+1,$ $A_k = \frac{k(k+1)}{L},$ and $\lambda_k = \frac{B_k}{A_k L + B_k} = \frac{1}{k+1},$ we have that $\cc_k \leq \cc_0.$ By definition, and as $A_0 = 0,$ we have that $\cc_0 = 0$. Thus, $\cc_k \leq 0,$ $\forall k \geq 1,$ and it follows that
\begin{align*}
    \|F(\vu_k)\|^2 &\leq \frac{B_k}{A_k}\innp{F(\vu_k), \vu_0 - \vu_k}\\
    &= \frac{L}{k}\big(\innp{F(\vu_k), \vu^* - \vu_k} + \innp{F(\vu_k), \vu_0 - \vu^*} \big)\\
    &\leq \frac{L}{k}\Big( - \frac{1}{L}\|F(\vu_k)\|^2 + \|F(\vu_k)\|\|\vu_0 - \vu^*\|\Big),
\end{align*}
where the last inequality is by Eq.~\eqref{eq:minty-gap} and Cauchy-Schwarz. To complete the proof, it remains to rearrange the last inequality and divide both sides by $\|F(\vu_k)\|.$
\end{proof}
%
\subsection{Lower Bounds for Cocoercive Operators}\label{sec:min-max-lb}
In this section, we provide a lower bound that applies to the class of algorithms that construct their iterates as the sum of an initial point and a linear combination of the cocoercive operator $F:\rr^d\to\rr^d$ evaluated at any of the points seen up to the current iteration. In particular, given a $\frac{1}{L}$-cocoercive operator $F: \rr^d \to \rr^d,$ an algorithm's iterate $\vu_k$ at iteration $k$ can be expressed as 
\begin{equation}\label{eq:gen-algos-lb}
    \vu_k = \vu_{0} - \sum_{i=0}^{k-1}\beta_{i, k}F(\vu_i),
\end{equation}
where $\beta_{i, k}$ are real coefficients that can depend on $L$ but are otherwise independent of $F$. To state the lower bound, we use $\cF_{L, D}$ to denote the class of problems with $\frac{1}{L}$-cocoercive operators $F$ that satisfy $\|\vu^* - \vu_0\| \leq D,$ where $\vu_0 \in \rr^d$ is an arbitrary  initial point and $\vu^*$ is such that $F(\vu^*) = \zeros.$ We  assume w.l.o.g.~that $d$ is even.

To derive the lower bound, we use the framework developed in~\cite{arjevani2016lower,arjevani2016iteration}. To make use of this framework, which relies on the use of Chebyshev polynomials, it is necessary to construct hard instances corresponding to linear operators $F(\vu) = \mA \vu + \vb,$ where $\mA \in \rr^{d \times d}$ and $\vb \in \rr^d.$ We note that such an approach was also used in~\cite{golowich2020last} for the class of monotone Lipschitz operators. However, here we aim to provide a lower bound for the more restricted class of cocoercive operators, which necessitates a separate construction. In particular, the monotone operator from the lower bound instance used in~\cite{golowich2020last} is not cocoercive as it corresponds to a bilinear function; in fact, it satisfies $\innp{F(\vu)- F(\vv), \vu - \vv} = 0,$ $\forall \vu, \vv \in \rr^d.$ 

Before delving into the technical details of our lower bound, we first provide definitions and supporting claims from~\cite{arjevani2016lower} that are needed for stating and proving it. A useful definition is that of 1-SCLI algorithms, which allows abstracting algorithms of the form from Eq.~\eqref{eq:gen-algos-lb} through the lens of Chebyshev polynomials. Here, we adopt the terminology from~\cite{golowich2020last}, which somewhat blurs the lines between various definitions (of stationary, oblivious, $p$-SCLI) algorithm types from~\cite{arjevani2016iteration,arjevani2016lower}, but provides perhaps the simplest way of stating the results.  

\begin{definition}[1-SCLI Algorithms]\label{def:1-SCLI-algo}
 An optimization algorithm $\cA$ acting on the class of linear operators $F:\rr^d \to \rr^d$ of the form $F(\vu) = \mA \vu + \vb$, where $\mA \in \rr^{d\times d},$ $\vb \in \rr^d,$ is said to be $1$-stationary  canonical linear iterative ($1$-SCLI) over $\rr^d$ if, given an initial point $\vu_0 \in \rr^d$, there exist mappings $C_0(\mathbf{A}), N(\mathbf{A}) : \rr^{d\times d} \to \rr^{d\times d}$ such that for all $k \geq 1$ the iterates of $\cA$ can be expressed as
 $$
    \vu_k = C_0(\mA) \vu_{k-1} + N(\mA)\vb. 
 $$
\end{definition}
Observe here that Definition~\ref{def:1-SCLI-algo} imposes no restrictions on what kind of mappings $C_0$ and $N$ can be. In particular, they can be polynomials of an arbitrary degree. This is important because choosing polynomials of degree $K$ would allow us to emulate arbitrary algorithms of the form from Eq.~\eqref{eq:gen-algos-lb} run over $K$ iterations, as $F$ is assumed to be linear (this observation is typically used in the analysis of the classical conjugate gradient method; see, e.g.,~\cite[Chapter 5]{nocedal2006numerical}). On the other hand, restricting the degree of the polynomials would restrict the adaptivity of coefficients $\beta_{i, k}$, as $C_0, N$ remain fixed for all $k.$ In this context, both GDA and Halpern iteration (when restricted to be run over a fixed number $K$ of iterations) can be viewed as 1-SCLI algorithms, with the following crucial difference. For GDA with a fixed step size $\eta$, we have
$$
    \vu_k = (\mI - \eta\mA)\vu_{k-1} + \eta \vb,
$$
i.e., $C_0$ is of degree one and $N$ is of degree zero. On the other hand, for Halpern iteration, 
\begin{align}\label{eq:halpern-linear}
    \vu_k &= \lambda_k \vu_0 + (1-\lambda_k)\Big(\mI - \frac{2}{L}\mA\Big)\vu_{k-1} + (1-\lambda_k)\vb.
\end{align}
By recursively applying Eq.~\eqref{eq:halpern-linear} and rolling it down to zero, we get that $\vu_k$ can be expressed as $\vu_k = C_0(\mA)\vu_0 + N(\mA)\vb$ using $C_0$ that is a polynomial of degree $k$ and $N$ that is a polynomial of degree $k-1.$ In other words, we can view $k$ iterations of Halpern's algorithm as one iteration of a 1-SCLI algorithm, using polynomial maps $C_0$ and $N$ of suitably large degrees. 
This is crucial for understanding the statement of the lower bound, which will effectively tell us that GDA is iteration complexity-optimal among all algorithms of the form from Eq.~\eqref{eq:gen-algos-lb} that choose steps sizes $\beta_{i, k}$ independently of $k,$ while Halpern iteration is iteration complexity-optimal over all algorithms that are allowed to adapt $\beta_{i, k}$'s to $k.$

In the following, we further restrict our attention to operators $F$ corresponding to full-rank matrices $\mA.$ This is convenient because the optimal solution $\vu^*$ for which $F(\vu^*) = \zeros$ can be expressed in closed form as $\vu^* = -\mA^{-1}\vb.$ This allows us to relate the polynomials $C_0$ and $N$ under a minimal (and standard~\cite{arjevani2016iteration,arjevani2016lower,golowich2020last}) assumption that the 1-SCLI algorithms we consider are \emph{consistent} (or convergent). We note here that the consistency condition is not necessary; it is rather the case that the proof relies on the relationship between $C_0$ and $N$ from Eq.~\eqref{eq:consistency}, for which the natural consistency condition suffices. 
\begin{definition}[Consistency]\label{def:consistency}
A 1-SCLI algorithm $\cA$ is said to be consistent w.r.t.~a full-rank matrix $\mA$ if for any $\vb \in \rr^d$ we have that  $\vu_k$ converges to $\vu^* = -\mA^{-1}\vb$. A 1-SCLI algorithm is said to be consistent if it is consistent w.r.t.~any full-rank matrix $\mA.$
\end{definition}
The relationship between $C_0$ and $N$ for consistent algorithms is characterized by the following lemma.
\begin{lemma}[Consistency of 1-SCLI Algorithms~\cite{arjevani2016lower}]
If a 1-SCLI algorithm is consistent w.r.t.~$\mA$, then 
\begin{equation}\label{eq:consistency}
    C_0(\mA)=\mI+N(\mA)\mA. 
\end{equation}
\end{lemma}
Finally, the following auxiliary lemma will be useful when proving our lower bound.
\begin{lemma}[{\cite[Lemma 13]{golowich2020last}}]\label{lem:poly-lower-bound}
Let $L >0,$ let $p$ and $k$ be arbitrary but fixed  non-negative integers, and let $r(y)$ be a polynomial with real-valued coefficients of degree at most $p$, such that $r(0) = 1$. Then:
\begin{equation}\label{ineq:poly-lower-bound}
    \sup_{y\in (0,L]}y|r(y)|^k\geq\sup_{y\in[L/(20p^2k), L]}y|r(y)|^k > \frac{L}{40p^2k}. 
\end{equation}
\end{lemma}

We are now ready to state and prove our lower bound.
\begin{theorem}\label{thm:lower-bound-cocoercive-op}
Let $p, K$ be any two positive integer numbers, and let $L, D > 0.$ Then, for any consistent 1-SCLI algorithm $\cA$ acting on instances from $\cF_{L, D}$, initialized at $\vu_0 = \zeros$ and for which $N(\mA)$ is a matrix polynomial of degree at most $p -1$, 
$$
   \sup_{F \in \cF_{L, D}} \|F(\vu_K)\| \geq \frac{LD}{4p\sqrt{5K}}.
$$
\end{theorem}
\begin{proof}
Similar to~\cite{golowich2020last}, we start by showing that 
\begin{equation}\label{eq:z^(t)}
    \vu_k = (C_0(\mA)^k-\mI)\mA^{-1}\vb,
\end{equation} 
for all $k \geq 0$. This claim follows by induction on $k$. The base case $k = 0$ is immediate. For the inductive step, suppose that Eq.~\eqref{eq:z^(t)} holds for some $k - 1 \geq 0.$ Then by the definition of 1-SCLI algorithms and the consistency of $\cA$ (Definitions~\ref{def:1-SCLI-algo} and \ref{def:consistency}):
\begin{align*}
\vu_k &= C_0(\mA)\vu_{k-1} + N(\mA)\vb \\
&= C_0(\mA)(C_0(\mA)^{k-1} - \mI)\mA^{-1}\vb + (C_0(\mA) - \mI)\mA^{-1}\vb\\
&= (C_0(\mA)^k - \mI)\mA^{-1}\vb.
\end{align*}
Therefore, $F(\vu_k)$ can be expressed as
\begin{equation}\label{eq:F = C_0^t*b}
    F(\vu_k) = \mA\vu_k + \vb = C_0(\mA)^k \vb.
\end{equation}

Let us now specify the ``hard instance.'' Consider $F(\vu) = \mA \vu + \vb,$ where $\mA$ can be expressed as  $\mA = \big[\subalign{\eta\mI \: \alpha\mI \\ -\alpha\mI\: \eta\mI}\big]$ for some $\eta, \alpha \in \rr_+.$  (Observe that such an $F$ can be obtained from the convex-concave objective $\phi(\vx, \vy) = \frac{1}{2} \eta\vx^T\vx - \frac{1}{2} \eta\vy^T\vy + \alpha\vx^T\vy + \vb_1^T\vx-\vb_2^T\vy,$ where $ \vx,\vy,\vb_1,\vb_2 \in \rr^{{d}/{2}}$, $\vb = [{\vb_1}^T {\vb_2}^T]^T$.) 

Let us now argue that for suitably chosen $\eta, \alpha,$ we have that $F$ is $\frac{1}{L}$-cocoercive. Let $\vu = [\vx^T\, \vy^T]^T$, $\vub = [\vxb^T \,\vyb^T]^T$ be an arbitrary pair of vectors from $\rr^d,$ where $\vx, \vy, \vxb, \vyb \in \rr^{d/2}.$ Then
\begin{equation*}
    \innp{F(\vu) - F(\vub), \vu - \vub} = \eta \|\vu - \vub\|^2
\end{equation*}
and
\begin{equation*}
    \|F(\vu) - F(\vub)\|^2 = 
    (\eta^2+\alpha^2)\|\vu - \vub\|^2.
\end{equation*}
Hence, for $\eta^2 + \alpha^2 \leq L\eta,$ we have $\innp{F(\vu) - F(\vub), \vu - \vub} \geq \frac{1}{L}\|F(\vu) - F(\vub)\|^2,$ i.e., $F$ is $\frac{1}{L}$-cocoercive. 

To complete the proof, it remains to show that 
\begin{equation*}
    \sup_{F \in \cF_{L, D}}{\|F(\vu_K)\|} \geq \frac{LD}{p\sqrt{80K}}.
\end{equation*}
To do so, observe that by Eq.~\eqref{eq:F = C_0^t*b}, $\vu_0 = \zeros,$ and $\vu^* = -\mA^{-1}\vb,$ we have
\begin{equation*}
    \sup_{F \in \cF_{L, D}} \frac{\|F(\vu_K)\|^2}{\|\vu^* - \vu_0\|^2} \geq \sup_{\substack{\eta \in [0, L], \\ \alpha \in [0, \sqrt{L\eta - \eta^2}]}} \frac{\|C_0(\mA)^K\vb\|^2}{\|\mA^{-1}\vb\|^2}, 
\end{equation*}
where $\mA= \big[\subalign{\eta \mI \: \alpha\mI \\ -\alpha \mI\: \eta \mI}\big].$ 
Observe that the characteristic polynomial of $\mA = \big[\subalign{\eta \mI \: \alpha \mI \\ -\alpha \mI\: \eta \mI}\big]$ is:
$$
\mathrm{det}(\lambda \mI - \mathbf{A}) 
= ((\lambda-\eta)^2+\alpha^2)^{{d}/{2}}. 
$$
Hence, $\mA$ has eigenvalues:
$
\lambda_1 = \eta + \alpha i, \, \lambda_2 = \eta - \alpha i.
$ 
These conjugate eigenvalues have the same magnitude: $\sqrt{\eta^2 + \alpha^2}$. Accordingly, $\mathbf{A}^{-1}$ has eigenvalues: $
\lambda_1' = \frac{1}{\lambda_1},\quad\lambda_2' = \frac{1}{\lambda_2}$, which are also conjugate and equal in  magnitude. On the other hand, since $C_0(\mA) = \mI + N(\mA)\mA$, and, by assumption, $N(\mA)$ is a matrix polynomial of degree at most $p-1$ for some $p \in \mathbb{N}$ with real coefficients, $C_0(\mA)$ is a polynomial of $\mA$ with $C_0(\zeros_{d\times d}) = \mI$. Therefore, it can be expressed as: $$
C_0(\mA) = \mI + r_1\mA + r_2\mathbf{A}^2 + \dots + r_{p}\mA^{p},
$$
for some real-valued $r_1, r_2, r_3, \dots, r_{p}$. We denote the polynomial on complex field with the same real-valued coefficients as: $c_0(y) = 1 + r_1y + r_2y^2 + \dots + r_{p}y^{p}$. Then, by the spectral mapping theorem, the eigenvalues of $C_0(\mA)$ are: $c_0(\lambda_1)$ and $c_0(\lambda_2)$, which are again conjugate and have equal norms. Therefore, we have:
\begin{align*}
    \sup_{\substack{\eta\in [0, L]\\ \alpha \in [0, \sqrt{L\eta - \eta^2}]}}\frac{\norm{C_0(\mathbf{A})^K\vb}^2}{\norm{\mathbf{A}^{-1}\vb}^2}&= \sup_{\substack{\eta\in [0, L]\\ \alpha \in [0, \sqrt{L\eta - \eta^2}]}}\frac{|c_0(\lambda_1)|^{2K}\norm{\vb}^2}{\frac{1}{|\lambda_1|^2}\norm{\vb}^2} \\
    &=\sup_{\substack{\eta\in [0, L]\\ \alpha \in [0, \sqrt{L\eta - \eta^2}]}} (\eta^2 + \alpha^2)|c_0(\eta+\alpha i)|^{2K}.
\end{align*}

To derive the stated lower bound by applying Lemma~\ref{lem:poly-lower-bound}, we need to convert the above expression into a similar form: $\sup_{y\in (0,L]}y|r(y)|^k$. Here, we can observe the difference between the problem we are considering and the problem discussed in \cite{golowich2020last}. In~\cite{golowich2020last}, the eigenvalues are purely imaginary: $\nu i$ and $-\nu i$. As a result, the above expression can be written as: $\sup_{\nu\in (0,L]}\nu^2|c_0(\nu i)|^{2K}$. By taking the real part of this term, we get a smaller value $\sup_{\nu\in (0,L]}\nu^2|1-r_2\nu^2+r_4\nu^4-\dots+(-1)^{p'}r_{2p'}\nu^{2p'}|^{2K}$, where $p' = \lfloor p/2 \rfloor$. Thus, substituting $\nu^2$ with $y$, we get the equation that fits the inequality from Lemma~\ref{lem:poly-lower-bound}. However, the same strategy cannot be applied here since the real part of $(\eta^2 + \alpha^2)|c_0(\eta+\alpha i)|^{2K}$ is tangled up with $\alpha$ and $\eta$, hence making it impossible to get an equation of the form $y|r(y)|^k$ by simply taking its real part. 

Nevertheless, since we have the extra freedom of choosing $\alpha$, we can select $\alpha$ carefully to make the real part and imaginary part of $|c_0(\eta+\alpha i)|^{2K}$ separable, while keeping the constant $\eta^2 + \alpha^2$ large enough. In particular, this can be achieved for:
$$
\alpha^2 = L\eta-\eta^2.
$$
Observe that, as long as $\eta \leq L,$ we have $\alpha \in [0, \sqrt{L\eta - \eta^2}],$ as required in the bound above. 
It follows that:
\begin{align*}
    &\sup_{\substack{\eta\in [0, L]\\ \alpha \in [0, \sqrt{L\eta - \eta^2}]}} (\eta^2 + \alpha^2)|c_0(\eta + \alpha i)|^{2K} \\
    &\hspace{1in}\geq \sup_{\eta\in [0, L]}L\eta |c_0(\eta+\alpha i)|^{2K} \\
    &\hspace{1in}= \sup_{\eta\in [0, L]}L\eta |1+r_1(\eta+\alpha i) + \dots + r_{p}(\eta+\alpha i)^{p}|^{2K}.
\end{align*}
Observe that the factor $\alpha$ in the real terms of $(\eta+\alpha i)^j$ has only even order, therefore, $\mathrm{Re}(c_0(\eta + \alpha i))$ is a polynomial of $\eta$ and $\alpha^2$. Since $\alpha^2 = L\eta-\eta ^2$, it is actually a polynomial of $\eta$ exclusively with real-valued coefficients of degree at most $p$, which we denote as: $c'_0(\eta) = 1 + r'_1\eta + r'_2\eta^2 + \dots + r'_{p}\eta^{p}$. 
Therefore, we get:
\begin{align*}
    \sup_{\eta\in [0, L]}L\eta |c_0(\eta+\alpha i)|^{2K} 
    &\geq \sup_{\eta\in (0, L]}L\eta |\mathrm{Re}(c_0(\eta+\alpha i))|^{2K} \\
    &=\sup_{\eta\in (0, L]} L\eta|c_0'(\eta)|^{2K}. 
\end{align*}
By Lemma~\ref{lem:poly-lower-bound} and $\|\mA^{-1}\vb\| = D$, we now have:
\begin{align*}
    \sup_{F \in \cF_{L, D}} \frac{\|F(\vu_K)\|^2}{\|\vu^* - \vu_0\|^2} =  \sup_{F \in \cF_{L, D}} \frac{\|F(\vu_K)\|^2}{D^2} \geq \sup_{\eta\in (0, L]} L\eta|c_0'(\eta)|^{2K} \ge \frac{L^2}{80p^2K},
\end{align*}
and the claimed lower bound follows after rearranging the last inequality.
\end{proof}

The implications of Theorem~\ref{thm:lower-bound-cocoercive-op} are as follows. Among all algorithms that update their iterates as in Eq.~\eqref{eq:gen-algos-lb} and use constant (independent of the iteration count) step sizes $\beta_{i, k},$ GDA is iteration complexity-optimal for minimizing the norm of a cocoercive operator. This means that other standard methods such as the extragradient/mirror-prox~\cite{korpelevich1977extragradient,nemirovski2004prox} method, dual extrapolation~\cite{nesterov2007dual}, or the method of Popov~\cite{Popov1980}, which fall into the same category, cannot attain a convergence rate for minimizing $\|F(\cdot)\|$ that is faster than $1/\sqrt{k}.$ Thus, choosing step sizes $\beta_{i, k}$ that depend on the iteration count is essential for achieving the faster $1/k$ rate of Halpern's algorithm. Furthermore, this rate is unimprovable for any of the typical iterative methods that take the form from Eq.~\eqref{eq:gen-algos-lb}. 

\section{Conclusion and Future Work}
We presented a general and unifying potential function-based framework for analyzing the convergence of first-order algorithms under the gradient norm criterion in the settings of convex and min-max optimization. The framework is intuitive in that it provides an interpretation of the mechanism driving the convergence as a trade-off between reducing the norm of the gradient and reducing some notion of an optimality gap. 

Many interesting questions for future work remain. In particular, our framework is primarily applicable to Euclidean setups. Thus, it is an intriguing question whether it is possible to generalize it to other normed spaces. We note that beyond the Euclidean setups, the only results with near-optimal convergence for $\ell_p$-normed spaces in the setting of convex optimization are those for $\ell_{\infty}$ (where an $\ell_\infty$ variant of gradient descent is optimal) and the very recent results for $p \in [1, 2]$ that are based on a regularization trick~\cite{diakonikolas2021complementary}. In a different direction, as conjectured in Section~\ref{sec:cvx-opt}, it appears that fixing either the number of iterations or the accuracy of the problem in advance is crucial for achieving near optimal rates in the case of convex objectives, even in Euclidean setups. Proving such a lower bound would be very interesting, as it would likely require completely new mathematical techniques. Finally, very little is known about the convergence in gradient norm in convex-concave min-max optimization setups, both from the aspect of algorithms and the lower bounds. In particular, we are not aware of any lower bounds outside of the Euclidean setup considered here, while, similar as in the case of convex optimization, the only near-optimal algorithm is based on a regularization trick and applies only to $p \in [1, 2]$~\cite{song2020optimistic}. 

\bibliographystyle{abbrv}
\bibliography{references.bib}

\appendix

\section{Sequence Growth for the Optimized Gradient Method}

This section provides a technical lemma used in the proof of Theorem~\ref{thm:OGM-G}.

\begin{lemma}\label{lemma:OGM-G-A_k-growth}
Let $\{\beta_{i, k}\}_{i\leq k}$, $\{a_k\}_{k\geq 0},$ $\{A_k\}_{k\geq 0}$ be the sequences of real numbers that for $k \in \{0, \dots, K\}$ satisfy $\beta_{k, k} = 1,$ $A_k = \sum_{i=0}^k a_i,$ and 
\begin{gather}
    \beta_{k, K-1} + \frac{a_{k+1}}{A_K} = \frac{A_{k+1}}{a_{k+1}}, \label{eq:OGM-G-cond-1-appx}\\
    A_{k+1}\beta_{j, k} = A_k \beta_{j, k-1} + a_{k+1}\Big(\beta_{j, K-1} + \frac{a_{j+1}}{A_K}\Big) + a_{j+1}\Big(\beta_{k, K-1} + \frac{a_{k+1}}{A_K}\Big). \label{eq:OGM-G-cond-2-appx}
\end{gather}
Then the sequence $\{A_k\}_{k \geq 0}$ can be chosen as
\begin{equation}\label{eq:A_k expression-appx}
    \begin{cases}
    A_k = 1, & \text{ if } k = K;\\
    A_{k} = A_{k+1}\big[1+\frac{1}{2}A_{k+1}-\frac{1}{2}\sqrt{A_{k+1}(4+A_{k+1})}\big], & \text{ if } 0 \leq k \leq K-1
    \end{cases}
\end{equation}
and $\frac{A_K}{A_0} \geq \frac{(K+2)^2}{4}.$
\end{lemma}
\begin{proof}
%
First, we show that the sequence $\{A_k\}^K_{k=0}$ with $A_K = 1$ and that satisfies Eq.~\eqref{eq:OGM-G-cond-1-appx} and Eq.~\eqref{eq:OGM-G-cond-2-appx} has the following recursive relationship between two successive terms:
\begin{equation}\label{eq:A_k recursive 1}
    \frac{1}{A_{k-1}} = \frac{1}{A_k}+\frac{A_k}{a_k}
\end{equation}
which is equivalent to:
\begin{equation}\label{eq:A_k recursive 2}
    \frac{A_{k-1}A_k}{a_k}=\frac{a_k}{A_k}.
\end{equation}
Solving for $A_{k-1},$ this relationship leads to Eq.~\eqref{eq:A_k expression-appx}. We prove the recursive relationship by induction on $k$. First, for the base case $k = K$, setting $k = K-1$ in Eq.~\eqref{eq:b_k,K-1}, we have: 
$\beta_{K-1,K-1} = \frac{A_K}{a_K}-\frac{a_K}{A_K}.$
Since we have set $A_K = 1$ and $\beta_{K-1,K-1} = 1$, it follows that:
$$\frac{a_K}{A_K} = \frac{A_K-a_K}{a_K} = \frac{A_{K-1}}{a_K} = \frac{A_KA_{K-1}}{a_K}$$
which coincides with Eq.~\eqref{eq:A_k recursive 2}.

Now assume that Eq.~\eqref{eq:A_k recursive 1} (equivalently, Eq.~\eqref{eq:A_k recursive 2}) holds for $k = K, K-1, \dots, n+1$, and consider $k = n$. Setting $k = n$, $j = n-1$ in Eq.~\eqref{eq:A_k,b_j,k}, we have:
$$
A_{n+1}\beta_{n-1,n} = A_n\beta_{n-1,n-1}+a_{n+1}\frac{A_n}{a_n}+a_n\frac{A_{n+1}}{a_{n+1}}
$$
Hence:
\begin{equation}\label{eq:A_n+1 b_n-1,n}
    A_{n+1}\beta_{n-1,n} = A_n+a_{n+1}\frac{A_n}{a_n}+a_n\frac{A_{n+1}}{a_{n+1}}
\end{equation}
It turns out that we can express $A_{n+1}\beta_{n-1,n}$ using $A_k$ for $k$ ranging from $n+1$ to $K$. Let $k=\ell$, $\ell = n+1, n+2, \cdots, K-1$ and $j=n-1$ in Eq.~\eqref{eq:A_k,b_j,k}; then, we get:
$$
A_{\ell} \beta_{n-1, l-1} = A_{\ell+1}\beta_{n-1, \ell} - a_{\ell+1} \frac{A_{n}}{a_{n}} - a_{n} \frac{A_{\ell+1}}{a_{\ell+1}}.
$$
This is a recursive relation between $A_{\ell} \beta_{n-1, \ell-1}$ and $A_{\ell+1}\beta_{n-1, \ell}$. Applying this relation recursively from $\ell = n+1$ to $\ell = K-1$, we get:
\begin{align*}
    A_{n+1}\beta_{n-1,n}&=A_{K}\beta_{n-1,K-1}-\frac{A_n}{a_n}(a_{n+2}+\cdots+a_{K})-a_n\big(\frac{A_{n+2}}{a_{n+2}}+\cdots+\frac{A_K}{a_K}\big)\\
    &=A_K\big(\frac{A_n}{a_n}-\frac{a_n}{A_K}\big)-\frac{A_n}{a_n}\sum_{\ell=n+1}^{K-1}(A_{\ell+1}-A_{\ell})-a_n\sum_{\ell=n+1}^{K-1}\big(\frac{1}{A_{\ell}}-\frac{1}{A_{\ell+1}}\big)\\
    &=A_K\big(\frac{A_n}{a_n}-\frac{a_n}{A_K}\big)-\frac{A_n}{a_n}(A_K-A_{n+1})-a_n\big(\frac{1}{A_{n+1}}-\frac{1}{A_K}\big)\\
    &=\frac{A_nA_{n+1}}{a_n}-\frac{a_n}{A_{n+1}}.
\end{align*}
The second equation is valid due to our inductive hypothesis for $k = n+2, n+3, \dots, K$. To derive the last equation, we use that $A_K = 1$, by the lemma assumption. Plugging the above equation into Eq.~\eqref{eq:A_n+1 b_n-1,n}, we get:
$$
\frac{A_nA_{n+1}}{a_n} = A_n\frac{a_n+a_{n+1}}{a_n}+a_n\big(\frac{A_{n+1}}{a_{n+1}}+\frac{1}{A_{n+1}}\big). 
$$
Using the assumption that $\frac{1}{A_{n}} = \frac{1}{A_{n+1}}+\frac{A_{n+1}}{a_{n+1}}$,
we obtain Eq.~\eqref{eq:A_k recursive 2} for $k = n$, completing the inductive argument.

The recursive relationship between $A_k$ and $A_{k+1}$ from Eq.~\eqref{eq:A_k expression} be equivalently written as
\begin{equation*}
    \frac{1}{A_k} = \frac{1}{4}\Big(2+\frac{4}{A_{k+1}}+2\sqrt{1+\frac{4}{A_{k+1}}}\Big)
\end{equation*}
Denote $D_n = \frac{1}{A_{K-n}}$ for $n \in \{0, \dots, K\}.$ Then
\begin{equation}\label{eq:expression D_n}
    D_n=\frac{1}{2}+D_{n-1}+\sqrt{D_{n-1}+\frac{1}{4}}.
\end{equation}
We prove by induction that: 
\begin{equation}\label{eq:lower bound D_n}
    D_n\geq\frac{(n+2)^2}{4}.
\end{equation}
As $D_0 = \frac{1}{A_K} = 1,$ $D_0 = \frac{(0+2)^2}{4}$ holds by definition. Now suppose it also for some $n = j$, $0 \leq j \leq K-1$. Then:
\begin{align*}
    D_{j+1} &= \frac{1}{2}+D_j+\sqrt{D_j+\frac{1}{4}}\\
    &\geq\frac{1}{2}+\frac{1}{4}(j+2 + 1 - 1)^2+\frac{1}{2}\sqrt{(j+2)^2}\\
    &= \frac{1}{2} + \frac{1}{4}(j+3)^2 - \frac{1}{2}(j+3) + \frac{1}{4} + \frac{1}{2}(j+2)\\
    &>\frac{1}{4}(j+3)^2
\end{align*}
Thus, $D_K = \frac{1}{A_0} = \frac{A_K}{A_0} \geq \frac{(K+2)^2}{4},$ as claimed.
\end{proof}

\end{document}